\documentclass{article}
\usepackage{amsmath,amssymb,amsfonts}
\allowdisplaybreaks[4]
\usepackage{mathptmx}
\usepackage{cite}
\usepackage[dvipdfmx]{graphicx}
\usepackage[dvipdfmx]{color}
\newcommand{\fz}{\frac}
\newcommand{\prz}[2]{ \frac{\partial{#1}}{\partial{#2}} }
\newcommand{\pz}{\partial}
\newcommand{\lA}{\langle}
\newcommand{\rA}{\rangle}
\newcommand{\rarrow}{\rightarrow}

\newcommand{\ol}{\overline}
\renewcommand{\Omega}{\varOmega}
\renewcommand{\Gamma}{\varGamma}
\renewcommand{\epsilon}{\varepsilon}
\newcommand{\barO}{\bar{\Omega}}

\newtheorem{Def}{Definition}
\newtheorem{Thm}{Theorem}
\newtheorem{Prop}{Proposition}
\newtheorem{Lem}{Lemma}
\newtheorem{Hyp}{Hypothesis}
\newtheorem{Rmk}{Remark}
\newtheorem{Ex}{Example}

\newtheorem{proof}{\it Proof}%

\begin{document}
%%-----------------------------
%%      the top matter
%%-----------------------------
\title{Error estimates of a stabilized Lagrange-Galerkin scheme for the Navier-Stokes equations}
%\thanks{...}\thanks{...}% At most 5 thanks
%
%\author{Hirofumi~Notsu}
%\address{
%Waseda Institute for Advanced Study, Waseda University, 3-4-1, Ohkubo, Shinjuku, Tokyo 169-8555, Japan.\ \ \texttt{h.notsu@aoni.waseda.jp}
%}
%\author{Masahisa~Tabata}
%
%\address{
%Department of Mathematics, Waseda University, 3-4-1, Ohkubo, Shinjuku, Tokyo 169-8555, Japan.\ \ \texttt{tabata@waseda.jp}
%}
%
\author{Hirofumi~Notsu \thanks{
Waseda Institute for Advanced Study, Waseda University, 3-4-1, Ohkubo, Shinjuku, Tokyo 169-8555, Japan.\ \ \texttt{h.notsu@aoni.waseda.jp}
}
~and Masahisa~Tabata \thanks{
Department of Mathematics, Waseda University, 3-4-1, Ohkubo, Shinjuku, Tokyo 169-8555, Japan.\ \ \texttt{tabata@waseda.jp}
}
}
%
%\date{\today}
%
%\begin{abstract}
%Error estimates with optimal convergence orders are proved for a stabilized Lagrange-Galerkin scheme for the Navier-Stokes equations.
%The scheme is a combination of Lagrange-Galerkin method and Brezzi-Pitk\"aranta's stabilization method.
%It maintains the advantages of both methods;
%(i)~It is robust for convection-dominated problems and the system of linear equations to be solved is symmetric.
%(ii)~Since the P1 finite element is employed for both velocity and pressure, the number of degrees of freedom is much smaller than that of other typical elements for the equations, e.g., P2/P1.
%Therefore, the scheme is efficient especially for three-dimensional problems.
%The theoretical convergence orders are recognized numerically by two- and three-dimensional computations.
%\end{abstract}
%
%\begin{resume}
%\end{resume}
%
%\subjclass{65M12, 65M60, 65M25, 76D05}
%
%\keywords{Error estimates, the finite element method, the Lagrange-Galerkin method, pressure-stabilization, the Navier-Stokes equations.}
%
\maketitle
\begin{abstract}
\noindent
Error estimates with optimal convergence orders are proved for a stabilized Lagrange-Galerkin scheme for the Navier-Stokes equations.
The scheme is a combination of Lagrange-Galerkin method and Brezzi-Pitk\"aranta's stabilization method.
It maintains the advantages of both methods;
(i)~It is robust for convection-dominated problems and the system of linear equations to be solved is symmetric.
(ii)~Since the P1 finite element is employed for both velocity and pressure, the number of degrees of freedom is much smaller than that of other typical elements for the equations, e.g., P2/P1.
Therefore, the scheme is efficient especially for three-dimensional problems.
The theoretical convergence orders are recognized numerically by two- and three-dimensional computations.
\end{abstract}
%
%%-----------------------------
%%      your text
%%-----------------------------
%%%%%%%%%%%%%%%%%%%%%%%%%%%%%%%%%%%%%%%%%%%%%
\section{Introduction}
%%%%%%%%%%%%%%%%%%%%%%%%%%%%%%%%%%%%%%%%%%%%%
The purpose of this paper is to prove the stability and convergence of a stabilized Lagrange-Galerkin scheme for the Navier-Stokes equations.
The scheme is a combination of a Lagrange-Galerkin (LG) method and Brezzi-Pitk\"aranta's stabilization method~\cite{BP-1984}.  
It has been proposed by us in~\cite{NT-2008-JSIAM,N-2008-JSCES} and, to the best of our knowledge, it is one of the earliest works which combine the two methods, Lagrange-Galerkin and stabilization.
Optimal error estimates are shown for both velocity and pressure.  
\par
The LG method is a finite element method embracing the method of characteristics.    
The LG method has common advantages, robustness for convection-dominated problems and symmetry of the resulting matrix, which are desirable in scientific computation of fluid dynamics.  
Many authors have studied LG schemes for convection-diffusion problems~\cite{BNV-2006,DR-1982,ER-1981,PT-2010,RT-2002} and for the Navier-Stokes, Oseen and natural convection problems~\cite{AG-2000,BB-2011,BMMR-1997,JLL-2010,NT-2009-JSC,NT-Oseen,P-1982,Suli-1988}, see also the bibliography therein.  
The convergence analysis of LG schemes for the Navier-Stokes equations has been done by Pironneau \cite{P-1982} and improved by S\"uli \cite{Suli-1988}.  
The analysis has been extended to a higher-order time scheme by Boukir et al. ~\cite{BMMR-1997} and to the projection method by Achdou and Guermond \cite{AG-2000}.  
While in these analyses they use a stable element satisfying the conventional inf-sup condition~\cite{GR-1986}, we extend the convergence analysis to a stabilized LG scheme.
The reason to use the stabilized method is to reduce the number of degrees of freedom (DOF).  
In fact the cheapest P1 element is employed in our scheme for both velocity and pressure, which is  based on Brezzi-Pitk\"aranta's pressure-stabilization method.   
Hence, the number of DOF is much smaller than that of typical stable elements, e.g., P2/P1. 
As a result, the scheme leads to a small-size symmetric resulting matrix, which can be solved by a powerful linear solvers for symmetric matrices, e.g., minimal residual method (MINRES) \cite{Betal-1994,S-2003}. 
It is, therefore, efficient especially in three-dimensional computation.
\par
In LG schemes the position at the previous time $t^{n-1}$ of a particle is sought along the trajectory, which is governed by a system of ordinary differential equations.  
The position at $t^{n-1}$ of a particle at a point at $t^n$ is called upwind point of the point or foot of the trajectory arriving at the point.
While the system of ordinary differential equations is assumed to be solved exactly in \cite{AG-2000,Suli-1988},  approximate upwind points are computed explicitly without assuming the exact solvability of the ordinary differential equations in~\cite{BMMR-1997,P-1982}.
Therefore, we may say that the latter schemes are fully discrete.
Our scheme is also fully discrete since the approximate upwind points are simply obtained by the Euler method.  
In fully discrete schemes, however, it is not obvious that the approximate upwind points remain in the domain, which should be proved.
Such difficulty caused by the nonlinearity of the Navier-Stokes equations is overcome in the proof by mathematical induction, which has been developed in~\cite{BMMR-1997,Suli-1988}.
Thus, the stability and convergence with optimal error estimates are proved for the velocity in $H^1$-norm and for the pressure in $L^2$-norm~(Theorem~\ref{thm:main_results}) and for the velocity in $L^2$-norm~(Theorem \ref{thm:main_results_L2}) under the condition $\Delta t = O(h^{d/4})$, where $d$ is the dimension of the space.    
This condition is caused from the nonlinearity of the problem and it is not required for the Oseen problems ~\cite{NT-Oseen}.
A stabilized LG scheme with an $L^2$-type pressure-stabilization for the Navier-Stokes equations has been proposed in~\cite{JLL-2010}, where the exact solvability of the ordinary differential equations is assumed for upwind points.  
The optimal error estimates are proved under a strong stability condition $\Delta t = O(h^2)$ for $d=2$.  
\par
This paper is organized as follows.
Our stabilized LG scheme for the Navier-Stokes equations is presented in Section~\ref{sec:scheme}.
Main results on the stability and convergence with optimal error estimates are shown in Section~\ref{sec:main_results}, and they are proved in Section~\ref{sec:proofs}.
The theoretical convergence orders are recognized numerically by two- and three-dimensional computations in Section~\ref{sec:numerics}.
Conclusions are given in Section~\ref{sec:conclusions}. 
In Appendix two lemmas used in Section 4 are proved.
%
%
%
%
%
%
%
%
%
%
%
%
%
%
%
%%%%%%%%%%%%%%%%%%%%%%%%%%%%%%%%%%%%%%%%%%%%%
\section{A stabilized Lagrange-Galerkin scheme}\label{sec:scheme}
%%%%%%%%%%%%%%%%%%%%%%%%%%%%%%%%%%%%%%%%%%%%%
%
We prepare function spaces and notation to be used throughout the paper.
Let $\Omega$ be a bounded domain in $\mathbb{R}^d (d=2,3)$, $\Gamma\equiv\pz\Omega$ be the boundary of $\Omega$, and $T$ be a positive constant.
For an integer $m \ge 0$ and a real number $p\in [1,\infty]$ we use the {\rm Sobolev} spaces $W^{m,p}(\Omega)$, $W^{1,\infty}_0(\Omega)$, $H^m(\Omega) (=W^{m,2}(\Omega))$, $H^1_0(\Omega)$ and $H^{-1}(\Omega)$.
For any normed space $X$ with norm $\|\cdot\|_X$, we define function spaces $C([0,T]; X)$ and $H^m(0,T; X)$ consisting of $X$-valued functions in $C([0,T])$ and $H^m(0,T)$, respectively.
We use the same notation $(\cdot, \cdot)$ to represent the $L^2(\Omega)$ inner product for scalar-, vector- and matrix-valued functions.
The dual pairing between $X$ and the dual space $X^\prime$ is denoted by $\lA\cdot, \cdot\rA$.
The norms on $W^{m,p}(\Omega)^d$ and $H^m(\Omega)^d$ are simply denoted as
\begin{align*}
\|\cdot\|_{m,p} \equiv \|\cdot\|_{W^{m,p}(\Omega)^d},\quad \|\cdot\|_m \equiv \|\cdot\|_{H^m(\Omega)^d}~(= \|\cdot\|_{m,2})
\end{align*}
and the notation $\|\cdot\|_m$ is employed not only for vector-valued functions but also for scalar-valued ones.
We also denote the norm on $H^{-1}(\Omega)^d$ by $\|\cdot\|_{-1}$.
$L^2_0(\Omega)$ is a subspace of $L^2(\Omega)$ defined by
\begin{align*}
L^2_0(\Omega) \equiv \bigl\{q\in L^2(\Omega);\ (q, 1) = 0 \bigr\}.
\end{align*}
We often omit $[0,T]$, $\Omega$ and/or $d$ if there is no confusion, e.g., $C(L^\infty)$ in place of $C([0,T]; L^\infty(\Omega)^d)$.
For $t_0$ and $t_1\in\mathbb{R}$ we introduce function spaces
\begin{align*}
Z^m(t_0, t_1) \equiv \{ v\in H^j(t_0, t_1; H^{m-j}(\Omega)^d);\ j=0,\cdots,m,\ \|v\|_{Z^m(t_0, t_1)} < \infty \}
\end{align*}
and $Z^m \equiv Z^m(0, T)$, where the norm $\|v\|_{Z^m(t_0, t_1)}$ is defined by
\begin{align*}
\|v\|_{Z^m(t_0, t_1)} \equiv \biggl\{ \sum_{j=0}^m \|v\|_{H^j(t_0,t_1; H^{m-j}(\Omega)^d)}^2 \biggr\}^{1/2}.
\end{align*}
\par
We consider the Navier-Stokes problem; find $(u, p):$ $\Omega\times (0, T) \rarrow \mathbb{R}^d\times\mathbb{R}$ such that
\begin{subequations}\label{prob:NS}
\begin{align}
\fz{Du}{Dt} - \nabla \bigl\{ 2\nu D(u) \bigr\} + \nabla p & = f \qquad \ \mbox{in}\ \ \Omega\times (0,T), \\
\nabla \cdot u & = 0  \qquad \ \, \mbox{in}\ \ \Omega\times (0,T), \label{eq:div-free}\\
u & = 0 \qquad \ \mbox{on}\ \ \Gamma\times (0,T),\\
u & = u^0 \qquad \mbox{in}\ \ \Omega,\ \mbox{at}\ t=0,
\end{align}
\end{subequations}
where $u$ is the velocity, $p$ is the pressure, $f:\ \Omega\times(0, T)\rarrow \mathbb{R}^d$ is a given external force, $u^0:\ \Omega\rarrow \mathbb{R}^d$ is a given initial velocity, $\nu >0$ is a viscosity, $D(u)$ is the strain-rate tensor defined by
\begin{align*}
D_{ij}(u)\equiv \dfrac{1}{2}\left(\dfrac{\partial u_i}{\partial x_j}+\dfrac{\partial u_j}{\partial x_i}\right), \quad i, j = 1,\cdots, d,
\end{align*}
and $D/Dt$ is the material derivative defined by
\begin{align*}
\fz{D}{Dt}\equiv \prz{}{t} + u\cdot\nabla.
\end{align*}
\par
Letting $V \equiv H^1_0(\Omega)^d$ and $Q\equiv L^2_0(\Omega)$, we define bilinear forms $a$ on $V\times V$, $b$ on $V \times Q$ and $\mathcal{A}$ on $(V \times Q) \times (V \times Q)$ by
\begin{align*}
a(u,v) \equiv 2 \nu \bigl( D(u), D(v) \bigr),\quad 
b(v,q) \equiv -( \nabla\cdot v, q ),\quad
\mathcal{A} \bigl( (u,p), (v,q) \bigr) \equiv a(u,v)+b(v,p)+b(u,q),
\end{align*}
respectively.
Then, we can write the weak formulation of~\eqref{prob:NS}; find $(u, p):\ (0, T)\to V\times Q$ such that for $t\in (0,T)$
\begin{align}
\Bigl( \fz{Du}{Dt}(t), v \Bigr) + \mathcal{A}\bigl( (u,p)(t), (v,q) \bigr) & = (f(t), v),\quad \forall (v, q)\in V\times Q, \label{eq:NS_weak}
\end{align}
with $u(0) = u^0$.
\par
Let $\Delta t$ be a time increment and $t^n \equiv n\Delta t$ for $n\in\mathbb{N}\cup\{0\}$.
For a function $g$ defined in $\Omega\times (0,T)$ we denote generally $g(\cdot,t^n)$ by $g^n$.
Let $X: (0, T)\rightarrow\mathbb{R}^d$ be a solution of the system of ordinary differential equations,
\begin{align}
\frac{dX}{dt} = u(X, t).\label{eq:ode}
\end{align}
Then, it holds that
\begin{align*}
\frac{Du}{Dt} (X(t), t) = \frac{d}{dt} u \bigl( X(t), t \bigr),
\end{align*}
when $u$ is smooth.
Let $X(\cdot; x, t^n)$ be the solution of~\eqref{eq:ode} subject to an initial condition $X(t^n) = x$.
For a velocity $w: \Omega\to\mathbb{R}^d$ let $X_1(w,\Delta t): \Omega\to\mathbb{R}^d$ be a mapping defined by
\begin{align}
X_1(w,\Delta t)(x) \equiv x - w(x)\Delta t.
\label{def:X1}
\end{align}
Since the position $X_1(u^{n-1},\Delta t)(x)$ is an approximation of $X(t^{n-1};x,t^n)$ for $n\ge 1$, we can consider a first order approximation of the material derivative at $(x, t^n)$,
\begin{align*}
\frac{Du}{Dt} (x, t^n) & = \frac{d}{dt} u \bigl( X(t; x, t^n), t \bigr)\Bigr|_{t=t^n} = \frac{u^n - u^{n-1}\circ X_1(u^{n-1},\Delta t)}{\Delta t} (x) + O(\Delta t),
\end{align*}
where the symbol $\circ$ stands for the composition of functions,
\begin{align*}
(v\circ w) (x) \equiv  v\bigl( w(x) \bigr),
\end{align*}
for $v: \Omega \to \mathbb{R}^d$ and $w: \Omega \to \Omega$.
$X_1(w,\Delta t)(x)$ is called an upwind point of $x$ with respect to the velocity $w$.
The next proposition gives a sufficient condition to guarantee all upwind points are in $\Omega$.
\begin{Prop}[\!\!{\cite[Proposition~1]{RT-2002}}]\label{prop:X_in_omega}
Let $w\in W^{1,\infty}_0(\Omega)^d$ be a given function, and assume
\begin{align*}
\Delta t \|w\|_{1,\infty} < 1.
\end{align*}
Then, it holds that
\begin{align*}
X_1(w,\Delta t)(\Omega) = \Omega.
\end{align*}
\end{Prop}
\par
For the sake of simplicity we assume that $\Omega$ is a polygonal $(d=2)$ or polyhedral $(d=3)$ domain.
Let $\mathcal{T}_h=\{K\}$ be a triangulation of $\barO\ (= \bigcup_{K\in\mathcal{T}_h} K )$, $h_K$ be a diameter of $K\in\mathcal{T}_h$, and $h\equiv\max_{K\in\mathcal{T}_h}h_K$ be the maximum element size.
Throughout this paper we consider a regular family of triangulations $\{\mathcal{T}_h\}_{h\downarrow 0}$ satisfying the inverse assumption~\cite{C-1978}, i.e., there exists a positive constant $\alpha_0$ independent of $h$ such that
\begin{align}
\fz{h}{h_K} \le \alpha_0,\quad \forall K\in \mathcal{T}_h,\quad \forall h.
\label{ieq:inverse_assumption}
\end{align}
We define function spaces $X_h$, $M_h$, $V_h$ and $Q_h$ by
\begin{align*}
X_h \equiv \{ v_h\in C(\barO)^d;\ v_{h|K} \in P_1(K)^d,\ \forall K \in\mathcal{T}_h \},\quad
M_h \equiv \{ q_h\in C(\barO);\ q_{h|K} \in P_1(K),\ \forall K \in\mathcal{T}_h \},
\end{align*}
$V_h \equiv X_h\cap V$ and $Q_h \equiv M_h \cap Q$, respectively, where $P_1(K)$ is the space of linear functions on $K\in\mathcal{T}_h$.
Let $N_T\equiv \lfloor T/\Delta t \rfloor$ be a total number of time steps, $\delta_0$ be a positive constant and $(\cdot,\cdot)_K$ be the $L^2(K)^d$ inner product.
We define bilinear forms $\mathcal{C}_h$ on $H^1(\Omega) \times H^1(\Omega)$ and $\mathcal{A}_h$ on $(V \times H^1(\Omega)) \times (V \times H^1(\Omega))$ by
\begin{align}
\mathcal{C}_h (p, q) & \equiv \delta_0\sum_{K\in\mathcal{T}_h} h_K^2 ( \nabla p, \nabla q )_K,\notag\\
\mathcal{A}_h \bigl( (u,p), (v,q) \bigr) & \equiv a(u,v)+b(v,p)+b(u,q) - \mathcal{C}_h(p, q).\label{def:Ah}
\end{align}
The bilinear form~$\mathcal{C}_h$ has been originally introduced in~\cite{BP-1984} in order to stabilize the pressure.
\par
Suppose $f\in C([0,T]; L^2(\Omega)^d)$ and $u^0\in V$.
Let an approximate function $u^0_h\in V_h$ of $u^0$ be given.
Our stabilized LG scheme for~\eqref{prob:NS} is to find $\{(u_h^n,\, p_h^n)\}_{n=1}^{N_T}\subset V_h \times Q_h$ such that for $n=1,\cdots, N_T$
\begin{align}
\Bigl(\fz{u_h^n-u_h^{n-1}\circ X_1(u_h^{n-1},\Delta t)}{\Delta t}, v_h\Bigr) + \mathcal{A}_h \bigl( (u^n_h, p^n_h), (v_h, q_h) \bigr) = (f^n, v_h),\quad
\forall (v_h, q_h) \in V_h\times Q_h.
\label{scheme_NS}
\end{align}
\begin{Rmk}\label{rmk:scheme}
(i)~By expanding $u_h^n$ and $p_h^n$ in terms of a basis of $V_h$ and $Q_h$, scheme~\eqref{scheme_NS} leads to a symmetric matrix of the form
\begin{align*}
\begin{pmatrix}
A & \ B^T \\
B & -C
\end{pmatrix}
,
\end{align*}
where $A$, $B$ and $C$ are sub-matrices derived from $\tfrac{1}{\Delta t}(u^n_h, v_h) + a(u^n_h, v_h)$, $b(u_h^n, q_h)$ and $\mathcal{C}_h(p_h^n, q_h)$, respectively, and the superscript ``\,$T$'' stands for the transposition. \smallskip\\
(ii)~The matrix is independent of time step $n$ and regular.
The regularity is derived from the fact that $(u_h^n, p_h^n)=(0, 0)$ when $u_h^{n-1}=f^n=0$ since we have
\begin{align*}
\fz{1}{\Delta t}\|u_h^n\|_0^2 + 2\nu \|D(u_h^n)\|_0^2 + \delta_0 \sum_{K\in\mathcal{T}_h} h_K^2 \|\nabla p_h^n\|_{L^2(K)^d}^2 = 0
\end{align*}
by substituting $(u_h^n, -p_h^n)\in V_h\times Q_h$ into $(v_h, q_h)$  in~\eqref{scheme_NS}. \smallskip\\
(iii)~There exists a unique solution $(u_h^n, p_h^n)$ if $X_1(u_h^{n-1},\Delta t)$ maps $\Omega$ into $\Omega$.
The condition is ensured if $\Delta t \|u_h^{n-1}\|_{1,\infty} < 1$, cf. Proposition~\ref{prop:X_in_omega}.
\end{Rmk}
%
%
%
%
%
%
%
%
%
%
%%%%%%%%%%%%%%%%%%%%%%%%%%%%%%%%%%%%%%%%%%%%%
\section{Main results}\label{sec:main_results}
%%%%%%%%%%%%%%%%%%%%%%%%%%%%%%%%%%%%%%%%%%%%%
In this section we show the main results, conditional stability and optimal error estimates of scheme~\eqref{scheme_NS}, which are proved in Section~\ref{sec:proofs}.
\par
We use the following norms and a seminorm,
$\|\cdot\|_{V_h}\equiv \|\cdot\|_V \equiv \|\cdot\|_1$, 
$\|\cdot\|_{Q_h} \equiv \|\cdot\|_Q \equiv \|\cdot\|_0$, 
\begin{align*}
\|u\|_{l^{\infty}(X)} &\equiv \max_{n=0,\cdots, N_T} \|u^n\|_{X},
&
\|u\|_{l^{2}_m(X)} &\equiv \biggl\{ \Delta t\sum_{n=1}^{m} \|u^n\|_X^2 \biggr\}^{1/2},
&
\|u\|_{l^{2}(X)} &\equiv \|u\|_{l^{2}_{N_T}(X)},
&
|p|_h & \equiv \biggl\{ \sum_{K\in \mathcal{T}_h} h_K^2 (\nabla p, \nabla p)_K \biggr\}^{1/2},
\end{align*}
for $m\in\{1,\cdots,N_T\}$ and $X=L^\infty(\Omega)$, $L^2(\Omega)$ and $H^1(\Omega)$.
$\ol{D}_{\Delta t}$ is the backward difference operator defined by
\begin{align*}
\ol{D}_{\Delta t} u^n \equiv \fz{u^n - u^{n-1}}{\Delta t}.
\end{align*}
\begin{Def}[Stokes projection]\label{def:StokesPrj}
For $(w,r) \in V \times Q$ we define the Stokes projection $(\hat{w}_h, \hat{r}_h) \in V_h\times Q_h$ of $(w, r)$ by
\begin{align}
\mathcal{A}_h\bigl( (\hat{w}_h, \hat{r}_h), (v_h, q_h) \bigr) = \mathcal{A} \bigl( (w, r), (v_h, q_h) \bigr), \quad \forall (v_h, q_h) \in V_h\times Q_h.
\label{eq:StokesPrj}
\end{align}
\end{Def}
\begin{Hyp}\label{hyp:regularity}
The solution $(u, p)$ of~\eqref{eq:NS_weak} satisfies
$u\in C([0,T]; W^{1,\infty}(\Omega)^d)\cap Z^2 \cap H^1(0,T; V\cap H^2(\Omega)^d)$
and
$p\in H^1(0,T;Q\cap H^1(\Omega))$.
\end{Hyp}
\begin{Thm}\label{thm:main_results}
Suppose Hypothesis~\ref{hyp:regularity} holds.
Then, there exist positive constants $h_0$ and $c_0$ independent of $h$ and $\Delta t$ such that, for any pair~$(h, \Delta t)$,
\begin{align}
h\in (0,h_0],\quad \Delta t\le c_0 h^{d/4},
\label{condition:h_dt}
\end{align}
the following hold.
\smallskip\\
(i)~Scheme~\eqref{scheme_NS} with $u_h^0$, the first component of the Stokes projection of $(u^0, 0)$ by~\eqref{eq:StokesPrj}, has a unique solution~$(u_h, p_h)=\{(u_h^n, p_h^n)\}_{n=1}^{N_T}\subset V_h\times Q_h$.
\medskip\\
(ii)~It holds that
\begin{align}
\|u_h\|_{l^\infty(L^\infty)}\le \|u\|_{C(L^\infty)}+1.
\label{ieq:stability}
\end{align}
(iii)~There exists a positive constant~$\bar{c}$ independent of $h$ and $\Delta t$ such that
\begin{align}
\|u_h-u\|_{l^\infty(H^1)},\ \ \Bigl\|\ol{D}_{\Delta t}u_h-\prz{u}{t}\Bigr\|_{l^2(L^2)},\ \ \|p_h-p\|_{l^2(L^2)} \le \bar{c} (\Delta t + h).
\label{ieq:main_results}
\end{align}
\end{Thm}
\begin{Rmk}
Since the initial pressure $p^0$ is not given in~\eqref{prob:NS}, we cannot practice the Stokes projection of $(u^0, p^0)$.
That is the reason why we employ the Stokes projection of $(u^0, 0)$ and set the first component as $u_h^0$.
This choice is sufficient for the error estimates~\eqref{ieq:main_results} and also~\eqref{ieq:L2} in Theorem~\ref{thm:main_results_L2} below.
\end{Rmk}
\begin{Hyp}\label{hyp:L2}
The Stokes problem is regular, i.e., for any $g\in L^2(\Omega)^d$ the solution $(w, r)\in V\times Q$ of the Stokes problem,
\begin{align*}
\mathcal{A}\bigl( (w, r), (v, q) \bigr) = (g, v), \quad \forall (v, q) \in V\times Q,
\end{align*}
belongs to $H^2(\Omega)^d\times H^1(\Omega)$ and the estimate
\begin{align*}
\|w\|_2 + \|r\|_1 \le c_R\|g\|_0
\end{align*}
holds, where $c_R$ is a positive constant independent of $g$, $w$ and $r$.
\end{Hyp}
\begin{Thm}\label{thm:main_results_L2}
Suppose Hypotheses~\ref{hyp:regularity} and~\ref{hyp:L2} hold.
Then, there exists a positive constant~$\tilde{c}$ independent of $h$ and $\Delta t$ such that
\begin{align}
\|u_h-u\|_{l^\infty(L^2)} \le \tilde{c} (\Delta t + h^2),
\label{ieq:L2}
\end{align}
where $u_h$ is the first component of the solution of~\eqref{scheme_NS} stated in Theorem~\ref{thm:main_results}-(i).
\end{Thm}
\begin{Rmk}
Hypothesis~\ref{hyp:L2} holds, e.g., if $\Omega$ is convex in $\mathbb{R}^2$, cf.~\cite{GR-1986}.
\end{Rmk}
%
%
%
%
%
%
%
%
%
%
%%%%%%%%%%%%%%%%%%%%%%%%%%%%%%%%%%%%%%%%%%%%%
\section{Proofs of Theorems~\ref{thm:main_results} and~\ref{thm:main_results_L2}}\label{sec:proofs}
%%%%%%%%%%%%%%%%%%%%%%%%%%%%%%%%%%%%%%%%%%%%%
%
We use $c$, $c_u$ and $c_{(u,p)}$ to represent the generic positive constants independent of the discretization parameters $h$ and~$\Delta t$.
$c_u$ and $c_{(u,p)}$ are constants depending on $u$ and $(u,p)$, respectively.
The symbol ``$\prime$ (prime)'' is sometimes put in order to distinguish between two constants, e.g.,~$c_u$ and~$c_u^\prime$.
%
%%%%%%%%%%%%%%%%%%%%%%%
\subsection{Preparations}
We recall some lemmas and a proposition, which are directly used in our proofs.
The next lemma is derived from Korn's inequality~\cite{DL-1976}.
\begin{Lem}\label{lem:Korn}
Let $\Omega$ be a bounded domain with a Lipschitz-continuous boundary.
Then, there exists a positive constant $\alpha_1$ and the following inequalities hold.
\begin{align}
\|D(v)\|_0 \le \|v\|_1 \le \alpha_1 \|D(v)\|_0,\qquad \forall v \in H^1_0(\Omega)^d.
\end{align}
\end{Lem}
\noindent
We use inverse inequalities and interpolation properties.
\begin{Lem}[\!\!\cite{C-1978}]\label{lem:inverse_inequality}
There exist positive constants $\alpha_{2i}$, $i=0,\cdots,4$, independent of $h$ and the following inequalities hold.
\begin{subequations}\label{ieq:inverse_interpolation_property}
\begin{align}
&& |q_h|_h & \le \alpha_{20} \|q_h\|_0, & \forall q_h & \in Q_h,&&
\label{ieq:q_h}\\
&&\|v_h\|_{0,\infty} & \le \alpha_{21} h^{-d/6}\|v_h\|_1, & \forall v_h & \in V_h,&&
\label{ieq:inverse_0_inf}\\
&&\|v_h\|_{1,\infty} & \le \alpha_{22} h^{-d/2}\|v_h\|_1, & \forall v_h & \in V_h,&&
\label{ieq:inverse_1_inf}\\
&&\|\Pi_h v\|_{0,\infty} & \le \|v\|_{0,\infty}, & \forall v & \in C(\bar{\Omega})^d,&&
\label{ieq:interpolation_0_inf}\\
&&\|\Pi_h v\|_{1,\infty} & \le \alpha_{23} \|v\|_{1,\infty}, & \forall v & \in W^{1,\infty}(\Omega)^d,&&
\label{ieq:interpolation_1_inf}\\
\phantom{MMMMMMMM}
&&\|\Pi_h v - v\|_1 & \le \alpha_{24} h \|v\|_2, & \forall v & \in H^2(\Omega)^d,&&
\phantom{MMMMMM}
\label{ieq:interpolation_1_first_order}
\end{align}
where $\Pi_h: C(\bar{\Omega})^d\to X_h$ is the Lagrange interpolation operator.
\end{subequations}
\end{Lem}
\begin{Rmk}\label{rmk:interpolation_property}
(i)~Although the inverse assumption~\eqref{ieq:inverse_assumption} is supposed throughout the paper, it is not required for the estimates~\eqref{ieq:q_h}, \eqref{ieq:interpolation_0_inf}, \eqref{ieq:interpolation_1_inf} and~\eqref{ieq:interpolation_1_first_order}.
The assumption that $\{\mathcal{T}_h\}_{h\downarrow 0}$ is regular is sufficient for them.
\ 
(ii)~The inverse inequality~\eqref{ieq:inverse_0_inf} is sufficient in this paper, while it is not optimal for $d=2$.
\ 
(iii)~We note $\alpha_{23} \ge 1$.
\end{Rmk}
\begin{Lem}[\!\!{\cite[Lemma~3.2]{FS-1991}}]\label{lem:stability_Ah}
There exists a positive constant $\alpha_{30}$ independent of $h$ such that for any $h$
\begin{align}
\inf_{(w_h,r_h) \in V_h\times Q_h} \sup_{(v_h,q_h)\in V_h\times Q_h} \fz{\mathcal{A}_h \bigl( (w_h,r_h), (v_h,q_h) \bigr)}{\|(w_h,r_h)\|_{V\times Q} \|(v_h,q_h)\|_{V\times Q}}\ge \alpha_{30}.
\label{ieq:stability_Ah}
\end{align}
\end{Lem}
\begin{Rmk}
Although the conventional inf-sup condition~\cite{GR-1986},
\begin{align*}
\inf_{q_h\in Q_h}\sup_{v_h\in V_h} \fz{b(v_h, q_h)}{\|v_h\|_1 \|q_h\|_0} \ge \beta^\ast >0,
\end{align*}
does not hold true for the pair of $V_h$ and $Q_h$, the P1/P1 finite element spaces, $\mathcal{A}_h$ satisfies the stability inequality~\eqref{ieq:stability_Ah} for this pair.
\end{Rmk}
\begin{Prop}[\!\!\cite{BD-1988}]\label{Prop:StokesPrj}
(i)~Suppose $(w, r) \in (V \cap H^2(\Omega)^d)\times(Q \cap H^1(\Omega))$.
Then, there exists a positive constant $\alpha_{31}$ independent of $h$ such that for any $h$ the Stokes projection $(\hat{w}_h, \hat{r}_h)$ of $(w, r)$ by~\eqref{eq:StokesPrj} satisfies
\begin{subequations}
\begin{align}
\|\hat{w}_h - w\|_1,\ \ \|\hat{r}_h - r\|_0,\ \ |\hat{r}_h - r|_h \le \alpha_{31} h \|(w, r)\|_{H^2\times H^1}.
\label{ieq:StokesPrj}
\end{align}
(ii)~Suppose Hypothesis~\ref{hyp:L2} additionally holds.
Then, there exists a positive constant $\alpha_{32}$ independent of $h$ such that for any~$h$
\begin{align}
\|\hat{w}_h - w\|_0 \le \alpha_{32} h^2 \|(w,r)\|_{H^2\times H^1}.
\label{ieq:StokesPrj_L2}
\end{align}
\end{subequations}
\end{Prop}
We recall some results concerning the evaluation of composite functions, which are mainly due to Lemma~4.5 in~\cite{AG-2000} and Lemma~1 in~\cite{DR-1982}.
In the next lemma $a$ and $b$ are any functions in $W^{1,\infty}_0(\Omega)^d$ satisfying
\begin{align*}
\Delta t \|a\|_{1,\infty},\ \Delta t \|b\|_{1,\infty} \le \delta_1,
\end{align*}
where $\delta_1$ is a constant stated in (i) of the following lemma.
We consider the mappings $X_1(a,\Delta t)$ and $X_1(b,\Delta t)$ defined in~\eqref{def:X1}.
\begin{Lem}\label{lem:v-vX}
(i)~There exists a constant $\delta_1 \in (0,1)$ such that
\begin{align}
J(x) \ge 1/2,\quad \forall x\in\Omega,
\label{ieq:Jacobian}
\end{align}
where $J$ is the Jacobian $\det (\pz X_1(a,\Delta t)/\pz x)$.
\medskip\\
(ii)~There exist positive constants $\alpha_{4i}$, $i=0,\cdots,3$, independent of $\Delta t$ such that the following inequalities hold.
\begin{subequations}
\begin{align}
&&&&\|g-g\circ X_1(a,\Delta t)\|_0 & \le \alpha_{40} \Delta t \|a\|_{0,\infty} \|g\|_1,& \forall g & \in H^1(\Omega)^d,&&
\label{ieq:v-vX_1}\\
&&&&\|g-g\circ X_1(a,\Delta t) \|_{-1} & \le \alpha_{41} \Delta t \|a\|_{1,\infty} \|g\|_0,& \forall g &\in L^2(\Omega)^d,&&
\label{ieq:v-vX_2}\\
&&&&\|g\circ X_1(b,\Delta t)-g\circ X_1(a,\Delta t)\|_0 & \le \alpha_{42} \Delta t \|b-a\|_0 \|g\|_{1,\infty},& \forall g & \in W^{1,\infty}(\Omega)^d,&&
\label{ieq:v-vX_3}\\
&&&&\|g\circ X_1(b,\Delta t)-g\circ X_1(a,\Delta t)\|_{0,1} & \le \alpha_{43} \Delta t \|b-a\|_0 \|g\|_1,& \forall g & \in H^1(\Omega)^d.&&
\label{ieq:v-vX_4}
\end{align}
\end{subequations}
\end{Lem}
\begin{proof}
Since $J_{ij}=\delta_{ij}-\Delta t\pz a_i/\pz x_j$, \eqref{ieq:Jacobian} is obvious.
It holds that for any $q\in [1,\infty)$, $p\in [1,\infty]$, $p^\prime$ with $1/p+1/p^\prime =1$ and $g\in W^{1,qp^\prime}(\Omega)^d$
\begin{align*}
\| g\circ X_1(b,\Delta t) - g \circ X_1(a,\Delta t) \|_{0,q} \le 2 \| X_1(b,\Delta t) - X_1(a,\Delta t) \|_{0,pq} \|\nabla g\|_{0,qp^\prime}
\end{align*}
from Lemma~4.5 in~\cite{AG-2000}, which implies~\eqref{ieq:v-vX_1}, \eqref{ieq:v-vX_3} and~\eqref{ieq:v-vX_4}.
For the proof of~\eqref{ieq:v-vX_2}, refer to Lemma~1 in~\cite{DR-1982}.
\end{proof}
%
%
%
%%%%%%%%%%%%%%%%%%%%%%%%%%%%%
\subsection{An estimate at each time step}
Let $(\hat{u}_h, \hat{p}_h)(t)\in V_h\times Q_h$ be the Stokes projection of $(u, p)(t)$ by~\eqref{eq:StokesPrj} for $t\in [0,T]$.
Letting
\begin{align*}
e_h^n\equiv u_h^n-\hat{u}_h^n,\quad \epsilon_h^n\equiv p_h^n-\hat{p}_h^n,\quad \eta(t)\equiv (u-\hat{u}_h)(t),
\end{align*}
we have for $n\ge 1$
\begin{align}
(\ol{D}_{\Delta t}e_h^n, v_h) + \mathcal{A}_h \bigl( (e_h^n,\epsilon_h^n), (v_h,q_h) \bigr) = \lA R_h^n, v_h\rA,\quad \forall (v_h, q_h)\in V_h\times Q_h,
\label{eq:e_epsilon_R}
\end{align}
where
\begin{align*}
R_h^n & \equiv \sum_{i=1}^4 R_{hi}^n, \\
R_{h1}^n & \equiv \fz{Du^n}{Dt} - \fz{u^n - u^{n-1} \circ X_1(u^{n-1},\Delta t)}{\Delta t},
& R_{h2}^n & \equiv \fz{1}{\Delta t}\Bigl\{ u^{n-1}\circ X_1(u_h^{n-1},\Delta t) - u^{n-1} \circ X_1(u^{n-1},\Delta t) \Bigr\}, \\
R_{h3}^n & \equiv \fz{1}{\Delta t}\Bigl\{ \eta^n - \eta^{n-1} \circ X_1(u_h^{n-1},\Delta t) \Bigr\},
& R_{h4}^n & \equiv -\fz{1}{\Delta t} \Bigl\{ e_h^{n-1} - e_h^{n-1} \circ X_1(u_h^{n-1},\Delta t) \Bigr\}.
\end{align*}
\eqref{eq:e_epsilon_R} is derived from~\eqref{scheme_NS}, \eqref{eq:StokesPrj} and~\eqref{eq:NS_weak}.
We note $e_h^0 = u_h^0-\hat{u}_h^0$ and set $\epsilon_h^0\equiv p_h^0-\hat{p}_h^0$, where $(u_h^0,p_h^0)$ is the Stokes projection of $(u^0,0)$ by \eqref{eq:StokesPrj}.
\par
Hereafter, let $\delta_1$ be the constant in Lemma~\ref{lem:v-vX}.
\begin{Prop}\label{prop:eh_epsh_Gronwall}
(i)~Let $(u^0, p^0)\in (H^2(\Omega)^d\cap V)\times (H^1(\Omega)\cap Q)$ be given and assume $\nabla\cdot u^0 =0$.
Then, there exists a positive constant $c_I$ independent of $h$ such that for any $h$
\begin{align}
\sqrt{\nu}\|D(e_h^0)\|_0+\sqrt{\fz{\delta_0}{2}} |\epsilon_h^0|_h \le c_I h.
\label{ieq:eh0_epsh0}
\end{align}
(ii)~Let $n\in \{1,\cdots,N_T\}$ be a fixed number and $u_h^{n-1} \in V_h$ be known.
Suppose the inequality
\begin{align}
\Delta t \|u_h^{n-1}\|_{1,\infty} \le \delta_1
\label{ieq:dt_uh_1inf}
\end{align}
holds.
Then, there exists a unique solution~$(u_h^n,p_h^n)\in V_h\times Q_h$ of~\eqref{scheme_NS}.
\medskip\\
(iii)~Furthermore, suppose Hypothesis~\ref{hyp:regularity} and the inequality
\begin{align}
\Delta t \|u\|_{C(W^{1,\infty})} \le \delta_1
\label{ieq:dt_u_1inf}
\end{align}
hold.
Let $p_h^{n-1}\in Q_h$ be known and suppose the equation
\begin{align}
b(u_h^{n-1},q_h)-\mathcal{C}_h(p_h^{n-1},q_h) = 0,\quad \forall q_h\in Q_h,\label{eq:b_Ch_n-1}
\end{align}
holds.
Then, it holds that
\begin{align}
&\ol{D}_{\Delta t} \Bigl( \nu\|D(e_h^n)\|_0^2 +\fz{\delta_0}{2}|\epsilon_h^n|_h^2 \Bigr) + \fz{1}{2}\|\ol{D}_{\Delta t}e_h^n\|_0^2 
\le A_1(\|u_h^{n-1}\|_{0,\infty}) \nu\|D(e_h^{n-1})\|_0^2 \notag\\
&\qquad
+ A_2(\|u_h^{n-1}\|_{0,\infty}) \Bigl\{ \Delta t\|u\|_{Z^2(t^{n-1},t^n)}^2 
+ h^2 \Bigl( \fz{1}{\Delta t} \|(u,p)\|_{H^1(t^{n-1},t^n; H^2\times H^1)}^2 +1 \Bigr) \Bigr\},
\label{ieq:eh_epsh_Gronwall}
\end{align}
where $A_i$, $i=1,2$, are functions defined by
\begin{align*}
A_i(\xi)\equiv c_i (\xi^2 +1 )
\end{align*}
and $c_i$, $i=1, 2$, are positive constants independent of $h$ and $\Delta t$.
They are defined by~\eqref{def:c1_c2} below.
\end{Prop}
\par
For the proof we use the next lemma, which is proved in Appendix~\ref{proof:lem_R}.
\begin{Lem}\label{lem:estimates_R}
Suppose Hypothesis~\ref{hyp:regularity} holds.
Let $n\in \{1,\cdots,N_T\}$ be a fixed number and $u_h^{n-1} \in V_h$ be known.
Then, under the conditions \eqref{ieq:dt_uh_1inf} and \eqref{ieq:dt_u_1inf} it holds that
\begin{subequations}
\begin{align}
\| R_{h1}^n \|_0 & \le c_u \sqrt{\Delta t} \|u\|_{Z^2(t^{n-1},t^n)}, \label{ieq:R1}\\
\| R_{h2}^n \|_0 & \le c_u \bigl( \|e_h^{n-1}\|_0 + h \|(u,p)^{n-1}\|_{H^2\times H^1} \bigr), \label{ieq:R2}\\
\| R_{h3}^n \|_0 & \le \fz{ch}{\sqrt{\Delta t}} (\|u_h^{n-1}\|_{0,\infty} + 1) \| (u,p) \|_{H^1(t^{n-1},t^n; H^2\times H^1)}, \label{ieq:R3}\\
\| R_{h4}^n \|_0 & \le c \|u_h^{n-1}\|_{0,\infty} \|e_h^{n-1}\|_1.\label{ieq:R4}
\end{align}
\end{subequations}
\end{Lem}
\begin{proof}[Proof of Proposition~\ref{prop:eh_epsh_Gronwall}]
We prove~(i).
Since $(u_h^0, p_h^0)$ and $(\hat{u}_h^0, \hat{p}_h^0)$ are the Stokes projections of $(u^0, 0)$ and $(u^0, p^0)$ by~\eqref{eq:StokesPrj}, respectively, we have
\begin{align*}
\|D(e_h^0)\|_0 & \le \|e_h^0\|_1 = \|u_h^0 - \hat{u}_h^0\|_1 \le \|u_h^0 - u^0\|_1 +\| u^0 - \hat{u}_h^0\|_1 \le 2\alpha_{31} h \|(u^0,p^0)\|_{H^2\times H^1}, \\
|\epsilon_h^0|_h & = |p_h^0-\hat{p}_h^0|_h \le | p_h^0 - 0 |_h + |\hat{p}_h^0-p^0|_h + |p^0|_h \le \alpha_{20} \bigl( \| p_h^0 - 0 \|_0 + \|\hat{p}_h^0 - p^0\|_0 \bigr) + h\|p^0\|_1 \\
& \le (2\alpha_{20}\alpha_{31} +1) h \| (u^0,p^0) \|_{H^2\times H^1},
\end{align*}
which imply~\eqref{ieq:eh0_epsh0} for $c_I \equiv \{ 2\sqrt{\nu}\alpha_{31}+\sqrt{\delta_0/2} (2\alpha_{20}\alpha_{31}+1) \} \|(u^0,p^0)\|_{H^2\times H^1}$.
\par
(ii) is obtained from~\eqref{ieq:dt_uh_1inf} and Remark~\ref{rmk:scheme}-(iii).
\par
We prove~(iii).
Substituting $(\ol{D}_{\Delta t}e_h^n, 0)$ into $(v_h, q_h)$ in~\eqref{eq:e_epsilon_R} and using the inequality~$(a^2 - b^2)/2 \le a (a - b)$, we have
\begin{align}
& \|\ol{D}_{\Delta t} e_h^n\|_0^2 + \ol{D}_{\Delta t} \bigl( \nu\|D(e_h^n)\|_0^2 \bigr) + b(\ol{D}_{\Delta t} e_h^n, \epsilon_h^n) \le \sum_{i=1}^4 \lA R_{hi}^n, \ol{D}_{\Delta t}e_h^n\rA,
\label{proof_proposition_1}
\end{align}
where it is noted that $X_1(u^{n-1},\Delta t)$ in $R_{hi}^n~(i=1, 2)$ maps $\Omega$ onto $\Omega$ by~\eqref{ieq:dt_u_1inf}.
From~\eqref{eq:b_Ch_n-1} and~\eqref{scheme_NS} with $v_h=0\in V_h$ it holds that
\begin{align}
b(u_h^k, q_h) - \mathcal{C}_h(p_h^k, q_h) = 0,\quad \forall q_h\in Q_h,\label{eq:uh_ph_n-1_n}
\end{align}
for $k=n-1$ and $n$.
Since $(\hat{u}_h^n, \hat{p}_h^n)$ is the Stokes projection of $(u^n, p^n)$ by~\eqref{eq:StokesPrj}, we have
\begin{align}
b(\hat{u}_h^k, q_h) - \mathcal{C}_h(\hat{p}_h^k, q_h) = b(u^k, q_h) =0,\quad \forall q_h\in Q_h,\label{eq:hat_uh_hat_ph_n-1_n}
\end{align}
for $k=n-1$ and $n$.
\eqref{eq:uh_ph_n-1_n} and~\eqref{eq:hat_uh_hat_ph_n-1_n} imply
\begin{align*}
b(\ol{D}_{\Delta t} e_h^n, q_h) - \mathcal{C}_h(\ol{D}_{\Delta t} \epsilon_h^n, q_h) = 0,\quad \forall q_h\in Q_h,
\end{align*}
which leads to
\begin{align}
-b(\ol{D}_{\Delta t} e_h^n, \epsilon_h^n) + \mathcal{C}_h(\ol{D}_{\Delta t} \epsilon_h^n, \epsilon_h^n) = 0
\label{proof_proposition_2}
\end{align}
by putting $q_h=-\epsilon_h^n\in Q_h$.
Adding~\eqref{proof_proposition_2} to~\eqref{proof_proposition_1} and using Lemma~\ref{lem:estimates_R} and the inequality $ab\le \beta a^2/2+b^2/(2\beta)\ (\beta >0)$,
we have
\begin{align}
&\|\ol{D}_{\Delta t}e_h^n\|_0^2 + \ol{D}_{\Delta t}\Bigl( \nu\|D(e_h^n)\|_0^2 +\fz{\delta_0}{2}|\epsilon_h^n|_h^2 \Bigr) 
\le \sum_{i=1}^4\lA R_{hi}^n, \ol{D}_{\Delta t}e_h^n\rA \notag\\
&\quad \le \Bigl(\sum_{i=1}^4\beta_i\Bigr) \|\ol{D}_{\Delta t}e_h^n\|_0^2
+ \fz{c_u\alpha_1^2}{\nu}\biggl( \fz{1}{\beta_2} + \fz{\|u_h^{n-1}\|_{0,\infty}^2}{\beta_4} \biggr) \nu\|D(e_h^{n-1})\|_0^2 \notag\\
&\qquad + c_u^\prime \biggl\{ \fz{\Delta t}{\beta_1}\|u\|_{Z^2(t^{n-1},t^n)}^2 + h^2 \biggl( \fz{1}{\beta_2}\|(u,p)\|_{C(H^2\times H^1)}^2 + \fz{\|u_h^{n-1}\|_{0,\infty}^2+1}{\beta_3\Delta t} \|(u,p)\|_{H^1(t^{n-1},t^n; H^2\times H^1)}^2 \biggr) \biggr\}
\label{proof_proposition_Gronwall}
\end{align}
for any positive numbers $\beta_i\ (i=1,\cdots,4)$, where the inequality~$\|e_h^{n-1}\|_0 \le \|e_h^{n-1}\|_1$ has been used.
By setting $\beta_i=1/8$ for $i=1,\cdots,4$ in~\eqref{proof_proposition_Gronwall} it holds that 
\begin{align*}
&\ol{D}_{\Delta t}\Bigl( \nu\|D(e_h^n)\|_0^2 +\fz{\delta_0}{2}|\epsilon_h^n|_h^2 \Bigr) + \fz{1}{2}\|\ol{D}_{\Delta t}e_h^n\|_0^2 
\le \fz{c_u}{\nu} \bigl( \|u_h^{n-1}\|_{0,\infty}^2 + 1 \bigr) \nu\|D(e_h^{n-1})\|_0^2 \\
&\qquad + c_{(u,p)} \Bigl\{ \Delta t\|u\|_{Z^2(t^{n-1},t^n)}^2 + h^2 \bigl( \|u_h^{n-1}\|_{0,\infty}^2+1\bigr) \Bigl( \fz{1}{\Delta t} \|(u,p)\|_{H^1(t^{n-1},t^n; H^2\times H^1)}^2 + 1 \Bigr) \Bigr\}.
\end{align*}
Putting
\begin{align}
c_1 \equiv c_u/\nu,
\qquad
c_2 \equiv c_{(u,p)},
\label{def:c1_c2}
\end{align}
we obtain~\eqref{ieq:eh_epsh_Gronwall}.
\end{proof}
%
%
%
%
%
%
%
%%%%%%%%%%%%%%%%%%%%%%%%%%%%%%%%%%%%%%%
\subsection{Proof of Theorem~\ref{thm:main_results}}
The proof is performed by induction through three steps.\medskip\\
\textit{Step~1} (Setting $c_0$ and $h_0$):\ 
Let $c_I$ and $A_i~(i=1,2)$ be the constant and the functions in Proposition~\ref{prop:eh_epsh_Gronwall}, respectively.
Let $a_1$, $a_2$ and $c_\ast$ be constants defined by
\begin{align*}
a_1 & \equiv A_1(\|u\|_{C(L^\infty)}+1),\quad a_2 \equiv A_2(\|u\|_{C(L^\infty)}+1), \\
c_\ast  & \equiv \fz{\alpha_1}{\sqrt{\nu}} \exp(a_1T/2) \max \Bigl\{ a_2^{1/2}\|u\|_{Z^2},  
a_2^{1/2} \bigl( \|(u,p)\|_{H^1(H^2\times H^1)} + T^{1/2} \bigr) + c_I\Bigr\}.
\end{align*}
We can choose sufficiently small positive constants $c_0$ and $h_0$ such that
\begin{subequations}\label{def:c0_h0}
\begin{align}
\alpha_{21} \Bigl\{ c_\ast (c_0h_0^{d/12}+h_0^{1-d/6}) + (\alpha_{24} + \alpha_{31}) h_0^{1-d/6}\|(u,p)\|_{C(H^2\times H^1)} \Bigr\} & \le 1, 
\label{def:c0_h0_Linf}\\
c_0 \Bigl[ \alpha_{22} \Bigl\{ c_\ast (c_0+h_0^{1-d/4}) + (\alpha_{24} + \alpha_{31}) h_0^{1-d/4}\|(u,p)\|_{C(H^2\times H^1)} \Bigr\} 
+ \alpha_{23} h_0^{d/4}\|u\|_{C(W^{1,\infty})} \Bigr] & \le \delta_1,
\label{def:c0_h0_W1inf}
\end{align}
\end{subequations}
since all the powers of $h_0$ are positive.
\medskip\\
\textit{Step~2} (Induction):\ 
For $n\in\{0,\cdots,N_T\}$ we set property P($n$),
\begin{align*}
\mbox{P($n$)}:
\left\{
\begin{aligned}
&
\!
\begin{aligned}
&
{\rm (a)}~\nu\|D(e_h^n)\|_0^2 +\fz{\delta_0}{2}|\epsilon_h^n|_h^2 + \fz{1}{2}\|\ol{D}_{\Delta t}e_h\|_{l^2_n(L^2)}^2 
\\ & \qquad
\le \exp (a_1 n\Delta t)
\Bigl[ \nu\|D(e_h^0)\|_0^2 +\fz{\delta_0}{2}|\epsilon_h^0|_h^2 + a_2\Bigl\{ \Delta t^2 \|u\|_{Z^2(0,t^n)}^2 
+ h^2 \bigl( \|(u,p)\|_{H^1(0,t^n; H^2\times H^1)}^2 + n\Delta t \bigr)
\Bigr\} \Bigr], 
\end{aligned}
\\
& {\rm (b)}~\|u_h^n\|_{0,\infty} \le \|u\|_{C(L^\infty)}+1,\\
& {\rm (c)}~\ \Delta t\|u_h^n\|_{1,\infty} \le \delta_1,
\end{aligned}
\right.
\end{align*}
where $\|\ol{D}_{\Delta t}e_h\|_{l^2_n(L^2)}$ vanishes for $n=0$.
P($n$)-(a) can be rewritten as
\begin{align}
x_n + \Delta t\sum_{i=1}^n y_i \le \exp( a_1n\Delta t ) \Bigl(x_0 + \Delta t \sum_{i=1}^n b_i \Bigr),
\label{ieq:proof_thm_P_n}
\end{align}
where
\begin{align*}
x_n & \equiv \nu\|D(e_h^n)\|_0^2 +\fz{\delta_0}{2}|\epsilon_h^n|_h^2,
&
y_i &\equiv \fz{1}{2}\|\ol{D}_{\Delta t}e_h^i\|_0^2,
&
b_i & \equiv a_2\Bigl\{ \Delta t\|u\|_{Z^2(t^{i-1},t^i)}^2 + h^2 \Bigl( \fz{1}{\Delta t} \|(u,p)\|_{H^1(t^{i-1},t^i; H^2\times H^1)}^2 +1 \Bigr) \Bigr\}.
\end{align*}
\par
We firstly prove the general step in the induction.
Supposing that P($n-1$) holds true for an integer $n\in\{1,\cdots,N_T\}$, we prove that P($n$) also does.
Since P($n-1$)-(c) is nothing but~\eqref{ieq:dt_uh_1inf}, there exists a unique solution $(u_h^n,p_h^n)\in V_h\times Q_h$ of equation~\eqref{scheme_NS} from Proposition~\ref{prop:eh_epsh_Gronwall}-(ii).
We prove P($n$)-(a).
\eqref{ieq:dt_u_1inf} holds from the estimate,
\begin{align*}
\Delta t \|u\|_{C(W^{1,\infty})} \le c_0h_0^{d/4} \|u\|_{C(W^{1,\infty})} \le c_0 \alpha_{23} h_0^{d/4} \|u\|_{C(W^{1,\infty})} \le \delta_1,
\end{align*}
from condition~\eqref{condition:h_dt}, Remark~\ref{rmk:interpolation_property}-(iii) and~\eqref{def:c0_h0_W1inf}.
\eqref{eq:b_Ch_n-1} is obtained from~\eqref{scheme_NS} for $n\ge 2$ and from the definition of $(u_h^0, p_h^0)$, i.e., the Stokes projection of $(u^0, 0)$ by~\eqref{eq:StokesPrj}, for $n=1$.
Hence \eqref{ieq:eh_epsh_Gronwall} holds from Proposition~\ref{prop:eh_epsh_Gronwall}-(iii).
Since the inequalities $A_i(\|u_h^{n-1}\|_{0,\infty}) \le a_i~(i=1,2)$ hold from P($n-1$)-(b),
\eqref{ieq:eh_epsh_Gronwall} implies
\begin{align*}
& \ol{D}_{\Delta t} x_n + y_n \le a_1 x_{n-1} + b_n, 
\end{align*}
which leads to
\begin{align}
x_n + \Delta t y_n \le \exp (a_1 \Delta t ) (x_{n-1} + \Delta t b_n)
\label{proof_thm1_1}
\end{align}
by $1 \le 1+a_1\Delta t \le \exp (a_1 \Delta t)$.
From P($n-1$)-(a), i.e.,
\begin{align}
x_{n-1} + \Delta t\sum_{i=1}^{n-1}y_i \le \exp\bigl\{ a_1(n-1)\Delta t\bigr\} \Bigl(x_0 + \Delta t \sum_{i=1}^{n-1}b_i \Bigr),
\label{proof_thm1_2}
\end{align}
it holds that
\begin{align}
x_n + \Delta t\sum_{i=1}^n y_i 
& \le \exp (a_1\Delta t) (x_{n-1} + \Delta t b_n) + \Delta t\sum_{i=1}^{n-1} y_i \quad \mbox{(by~\eqref{proof_thm1_1})}\notag\\
& \le \exp (a_1\Delta t) \Bigl( x_{n-1} + \Delta t\sum_{i=1}^{n-1} y_i + \Delta t b_n \Bigr) \notag\\
& \le \exp (a_1\Delta t) \Bigl[ \exp\bigl\{ a_1(n-1)\Delta t\bigr\} \Bigl(x_0 + \Delta t \sum_{i=1}^{n-1}b_i \Bigr) + \Delta t b_n \Bigr] \quad \mbox{(by~\eqref{proof_thm1_2})}\notag\\
& \le \exp (a_1n\Delta t) \Bigl(x_0 + \Delta t \sum_{i=1}^n b_i \Bigr),\notag
\end{align}
which is~\eqref{ieq:proof_thm_P_n}, i.e., P($n$)-(a).
\par
For the proofs of P($n$)-(b) and (c) we prepare the estimate of $\|e_h^n\|_1$.
From P($n$)-(a) and~\eqref{ieq:eh0_epsh0} it holds that
\begin{align}
& \nu\|D(e_h^n)\|_0^2 +\fz{\delta_0}{2}|\epsilon_h^n|_h^2 + \fz{1}{2}\|\ol{D}_{\Delta t}e_h\|_{l^2_n(L^2)}^2
\le \exp ( a_1T ) \Bigl[ c_I^2 h^2 
+ a_2\Bigl\{ \Delta t^2\|u\|_{Z^2}^2 + h^2 \bigl( \|(u,p)\|_{H^1(H^2\times H^1)}^2 + T \bigr) \Bigr\} \Bigr] \notag\\
& \quad \le \exp(a_1T) \Bigl[ a_2\Delta t^2\|u\|_{Z^2}^2 + h^2 \Bigl\{ a_2 \bigl( \|(u,p)\|_{H^1(H^2\times H^1)}^2 + T \bigr)+c_I^2\Bigr\} \Bigr]
\le \bigl\{ c_3 (\Delta t +h) \bigr\}^2,
\label{ieq:c3}
\end{align}
where
\begin{align*}
c_3 & \equiv \exp(a_1T/2) 
\max \Bigl\{ a_2^{1/2}\|u\|_{Z^2},\ a_2^{1/2} \bigl( \|(u,p)\|_{H^1(H^2\times H^1)} + T^{1/2} \bigr) + c_I\Bigr\}.
\end{align*}
\eqref{ieq:c3} implies
\begin{align}
&\|e_h^n\|_1 \le \alpha_1 \|D(e_h^n)\|_0 \le \fz{\alpha_1}{\sqrt{\nu}} c_3 (\Delta t + h) = c_\ast (\Delta t + h).\label{ieq:ehn_H1}
\end{align}
\par
We prove P($n$)-(b) and~(c).
Let $\Pi_h$ be the Lagrange interpolation operator stated in Lemma~\ref{lem:inverse_inequality}.
It holds that
\begin{align*}
\|u_h^n\|_{0,\infty} &\le \|u_h^n-\Pi_hu^n\|_{0,\infty} + \|\Pi_hu^n\|_{0,\infty} \le \alpha_{21} h^{-d/6}\|u_h^n-\Pi_hu^n\|_1 + \|\Pi_hu^n\|_{0,\infty} \\
& \le \alpha_{21} h^{-d/6} (\|u_h^n-\hat{u}_h^n\|_1 + \|\hat{u}_h^n-u^n\|_1 + \|u^n-\Pi_hu^n\|_1) + \|\Pi_hu^n\|_{0,\infty} \\
& \le \alpha_{21} h^{-d/6} \{ c_\ast (\Delta t+h) + \alpha_{31} h\|(u^n,p^n)\|_{H^2\times H^1} + \alpha_{24} h \|u^n\|_2 \} + \|u^n\|_{0,\infty} \quad \mbox{(by~\eqref{ieq:ehn_H1})}\\
& \le \alpha_{21} \{ c_\ast (c_0h_0^{d/12}+h_0^{1-d/6}) + (\alpha_{24} + \alpha_{31}) h_0^{1-d/6} \|(u,p)\|_{C(H^2\times H^1)} \} + \|u\|_{C(L^\infty)} \quad \mbox{(by \eqref{condition:h_dt})}\\
& \le 1 + \|u\|_{C(L^\infty)}\quad \mbox{(by~\eqref{def:c0_h0_Linf})},\\
\Delta t \|u_h^n\|_{1,\infty} & \le c_0h^{d/4} ( \|u_h^n-\Pi_hu^n\|_{1,\infty}+\|\Pi_hu^n\|_{1,\infty} )
\le c_0h^{d/4} ( \alpha_{22} h^{-d/2}\|u_h^n-\Pi_hu^n\|_1+\|\Pi_hu^n\|_{1,\infty} )\\
& \le c_0 \{ \alpha_{22} h^{-d/4} (\|u_h^n-\hat{u}_h^n\|_1 + \|\hat{u}_h^n-u^n\|_1 + \|u^n-\Pi_hu^n\|_1 ) +h^{d/4}\|\Pi_hu^n\|_{1,\infty} \}\\
& \le c_0 [ \alpha_{22} h^{-d/4} \{ c_\ast (\Delta t+h) + \alpha_{31}h \|(u^n,p^n)\|_{H^2\times H^1} + \alpha_{24} h \|u^n\|_2 \} + \alpha_{23} h^{d/4}\|u^n\|_{1,\infty} ]\\
& \le c_0 [ \alpha_{22} h^{-d/4} \{ c_\ast (c_0h^{d/4}+h) + (\alpha_{24} + \alpha_{31}) h \|(u^n,p^n)\|_{H^2\times H^1} \} + \alpha_{23} h^{d/4}\|u^n\|_{1,\infty} ]\\
& \le c_0 [ \alpha_{22} \{ c_\ast (c_0+h_0^{1-d/4}) + (\alpha_{24} + \alpha_{31}) h_0^{1-d/4} \|(u,p)\|_{C(H^2\times H^1)} \} + \alpha_{23} h_0^{d/4}\|u\|_{C(W^{1,\infty})} ]\\
& \le \delta_1\quad \mbox{(by~\eqref{def:c0_h0_W1inf})}.
\end{align*}
Therefore, P($n$) holds true.
\par
The proof of P($0$) is easier than that of the general step.
P($0$)-(a) obviously holds with equality.
P($0$)-(b) and (c) are obtained as follows.
\begin{align*}
\|u_h^0\|_{0,\infty} & \le \|u_h^0-\Pi_hu^0\|_{0,\infty} + \|\Pi_hu^0\|_{0,\infty} 
\le \alpha_{21} h^{-d/6} ( \|u_h^0 - u^0 \|_1 + \|u^0 - \Pi_hu^0\|_1 ) + \|\Pi_hu^0\|_{0,\infty} \\
& \le \alpha_{21} (\alpha_{31} + \alpha_{24}) h^{1-d/6} \|u^0 \|_2 + \|u^0\|_{0,\infty} 
\le 1 + \|u\|_{C(L^\infty)}\quad \mbox{(by~\eqref{def:c0_h0_Linf})},\\
\Delta t \|u_h^0\|_{1,\infty} & \le c_0h^{d/4} ( \|u_h^0-\Pi_hu^0\|_{1,\infty}+\|\Pi_hu^0\|_{1,\infty} )
\le c_0h^{d/4} ( \alpha_{22} h^{-d/2} \|u_h^0-\Pi_hu^0\|_1 +\|\Pi_hu^0\|_{1,\infty} ) \\
& \le c_0 \{  \alpha_{22} h^{-d/4} ( \|u_h^0 - u^0\|_1 + \|u^0-\Pi_hu^0\|_1 ) + h^{d/4}\|\Pi_hu^0\|_{1,\infty} \} \\
& \le c_0 \{  \alpha_{22} (\alpha_{31} + \alpha_{24}) h^{1-d/4} \|u^0\|_2 + \alpha_{23} h^{d/4}\|u^0\|_{1,\infty} \} 
\le \delta_1\quad \mbox{(by~\eqref{def:c0_h0_W1inf})}.
\end{align*}
Thus, the induction is completed.
\medskip\\
\textit{Step~3}:\ 
Finally we derive the results (i), (ii) and (iii) of the theorem.
The induction completed in the previous step implies that P($N_T$) holds true.
Hence we have~(i) and~(ii).
The first inequality of~\eqref{ieq:main_results} in (iii) is obtained from~\eqref{ieq:ehn_H1} and the estimate
\begin{align*}
\|u_h-u\|_{l^\infty(H^1)} \le \|e_h\|_{l^\infty(H^1)} + \|\eta\|_{l^\infty(H^1)} 
\le \|e_h\|_{l^\infty(H^1)} + \alpha_{31} h \|(u,p)\|_{C(H^2\times H^1)}.
\end{align*}
Combining the estimate
\begin{align*}
\Bigl\| \ol{D}_{\Delta t}u_h^n - \prz{u^n}{t} \Bigr\|_0 
& \le \|\ol{D}_{\Delta t}e_h^n\|_0 + \|\ol{D}_{\Delta t}\eta^n\|_0 + \Bigl\| \ol{D}_{\Delta t}u^n - \prz{u^n}{t} \Bigr\|_0 \\
& \le \|\ol{D}_{\Delta t}e_h^n\|_0 + \fz{\alpha_{31} h}{\sqrt{\Delta t}}\|(u,p)\|_{H^1(t^{n-1},t^n; H^2\times H^1)} + \sqrt{\fz{\Delta t}{3}} \Bigl\| \prz{^2u}{t^2} \Bigr\|_{L^2(t^{n-1},t^n; L^2)}
\end{align*}
with~\eqref{ieq:c3}, we get the second inequality of~\eqref{ieq:main_results}.
Here, for the estimates of the last two terms, we have used the equalities
\begin{align*}
\bigl( \ol{D}_{\Delta t} \eta^n \bigr) (x) = \int_0^1 \prz{\eta}{t} (x, t^{n-1}+s\Delta t) ds,\qquad
\Bigl( \ol{D}_{\Delta t}u^n - \prz{u^n}{t} \Bigr)(x) = -\Delta t \int_0^1 s\prz{^2u}{t^2} (x, t^{n-1}+s\Delta t) ds.
\end{align*}
\par
We prove the third inequality of~\eqref{ieq:main_results}.
It holds that
\begin{align}
\|\epsilon_h^n\|_0 & \le \|(e_h^n,\epsilon_h^n)\|_{V\times Q} \le \fz{1}{\alpha_{30}} \sup_{(v_h,q_h)\in V_h\times Q_h} \fz{\mathcal{A}_h\bigl( (e_h^n,\epsilon_h^n), (v_h,q_h) \bigr)}{\|(v_h,q_h)\|_{V\times Q}} 
= \fz{1}{\alpha_{30}} \sup_{(v_h,q_h)\in V_h\times Q_h} \fz{\lA R_h^n, v_h\rA -(\ol{D}_{\Delta t}e_h^n, v_h)}{\|(v_h,q_h)\|_{V\times Q}} 
\notag\\
& 
\le c_{(u,p)} \Bigl\{ \sqrt{\Delta t} \|u\|_{Z^2(t^{n-1},t^n)} 
+ h\Bigl( \fz{1}{\sqrt{\Delta t}} \|(u,p)\|_{H^1(t^{n-1},t^n; H^2\times H^1)} + 1 \Bigr) 
+ \|e_h^{n-1}\|_1 + \|\ol{D}_{\Delta t}e_h^n\|_0
\Big\} 
\label{ieq:epsilon_h_0}
\end{align}
for $n=1,\cdots,N_T$.
Here we have used Lemmas~\ref{lem:stability_Ah} and~\ref{lem:estimates_R}, the inequality~$\|e_h^{n-1}\|_0\le\|e_h^{n-1}\|_1$ and~\eqref{ieq:stability}.
We obtain the result by combining~\eqref{ieq:epsilon_h_0}, \eqref{ieq:c3} and the estimate 
\begin{align*}
\phantom{MMMMMMMa}
\|p_h-p\|_{l^2(L^2)} \le \|\epsilon_h\|_{l^2(L^2)} + \|\hat{p}_h-p\|_{l^2(L^2)} 
\le \|\epsilon_h\|_{l^2(L^2)} + \sqrt{T} \alpha_{31} h \|(u,p)\|_{C(H^2\times H^1)}.
\phantom{MMMMMMa}
%\qed
\end{align*}
%
%
%
%
%
%
%%%%%%%%%%%%%%%%%%%%%%%%%%%%%%%%%%%%%%%
\subsection{Proof of Theorem~\ref{thm:main_results_L2}}
We use the next lemma, which is proved in Appendix~\ref{proof:lem_R_L2}.
\begin{Lem}\label{lem:estimates_R_L2}
Suppose Hypotheses~\ref{hyp:regularity} and~\ref{hyp:L2} hold.
Let $n\in \{1,\cdots,N_T\}$ be a fixed number and $u_h^{n-1} \in V_h$ be known.
Then, under the conditions \eqref{ieq:dt_uh_1inf} and \eqref{ieq:dt_u_1inf} it holds that 
\begin{subequations}
\begin{align}
\| R_{h2}^n \|_0 & \le c_u \bigl( \|e_h^{n-1}\|_0 + h^2 \|(u,p)^{n-1}\|_{H^2\times H^1}\bigr),
\label{ieq:R2_L2}\\
\| R_{h3}^n \|_{V_h^\prime} 
& \le c_u \Bigl( \|(u,p)^{n-1}\|_{H^2\times H^1} \|e_h^{n-1}\|_0 + \fz{h^2}{\sqrt{\Delta t}} \|(u,p)\|_{H^1(t^{n-1},t^n; H^2\times H^1)} + h^2 \sum_{k=1}^2 \|(u,p)^{n-1}\|_{H^2\times H^1}^k \Bigr),
\label{ieq:R3_L2}\\
\| R_{h4}^n \|_{V_h^\prime} 
& \le c_u \bigl( 1+h^{-d/6}\|e_h^{n-1}\|_1 \bigr) \bigl( \|e_h^{n-1}\|_0 + h^2 \|(u,p)^{n-1}\|_{H^2\times H^1} \bigr).
\label{ieq:R4_L2}
\end{align}
\end{subequations}
\end{Lem}
\begin{proof}[Proof of Theorem~\ref{thm:main_results_L2}]
Since we have $\|e_h\|_{l^\infty(H^1)}\le c_\ast (\Delta t+h)\le c_\ast (c_0+h_0^{1-d/4})h^{d/4}$ from \eqref{ieq:ehn_H1} and \eqref{condition:h_dt}, \eqref{ieq:R4_L2} implies
\begin{align}
\| R_{h4}^n \|_{V_h^\prime} \le c_u c_\ast \bigl( \|e_h^{n-1}\|_0 + h^2 \|(u,p)^{n-1}\|_{H^2\times H^1} \bigr).
\label{ieq:R4_L2_1}
\end{align}
Substituting $(e_h^n, -\epsilon_h^n)$ into $(v_h,q_h)$ in~\eqref{eq:e_epsilon_R} and using Lemma~\ref{lem:Korn}, \eqref{ieq:R1}, \eqref{ieq:R2_L2}, \eqref{ieq:R3_L2}, \eqref{ieq:R4_L2_1} and the inequality $ab\le \beta a^2/2+b^2/(2\beta)\ (\beta >0)$, we have
\begin{align*}
& \ol{D}_{\Delta t} \Bigl( \fz{1}{2}\|e_h^n\|_0^2 \Bigr) + \fz{2\nu}{\alpha_1^2} \|e_h^n\|_1^2 + \delta_0|\epsilon_h^n|_h^2 \le \sum_{i=1}^4 \lA R_{hi}^n, e_h^n\rA \\
& \quad \le c_u\Bigl( \fz{1}{\beta_2} + \fz{\|(u,p)\|_{C(H^2\times H^1)}^2}{\beta_3} + \fz{c_\ast^2}{\beta_4} \Bigr) \|e_h^{n-1}\|_0^2 + \Bigl( \sum_{i=1}^4 \beta_i \Bigr) \|e_h^n\|_1^2 + c_u^\prime \biggl[ \fz{\Delta t}{\beta_1} \|u\|_{Z^2(t^{n-1},t^n)}^2 
\\
& \qquad + \fz{h^4}{\beta_3\Delta t}\|(u,p)\|_{H^1(t^{n-1},t^n; H^2\times H^1)}^2 + h^4 \Bigl\{ \Bigl( \fz{1}{\beta_2} + \fz{c_\ast^2}{\beta_4} \Bigr) \| (u,p) \|_{C(H^2\times H^1)}^2 + \fz{1}{\beta_3}\sum_{k=1}^2 \|(u,p)\|_{C(H^2\times H^1)}^{2k} \Bigr\} \biggr]
\end{align*}
for any $\beta_i > 0~(i=1,\cdots,4)$, where the inequality $\|e_h^n\|_0 \le \|e_h^n\|_1$ has been employed.
Hence, it holds that
\begin{align*}
\ol{D}_{\Delta t} \Bigl( \fz{1}{2}\|e_h^n\|_0^2 \Bigr) + \fz{\nu}{\alpha_1^2} \|e_h^n\|_1^2 & \le c_{(u,p)} \|e_h^{n-1}\|_0^2 
+ c_{(u,p)}^\prime \Bigl( \Delta t \|u\|_{Z^2(t^{n-1},t^n)}^2 +  \fz{h^4}{\Delta t}\|(u,p)\|_{H^1(t^{n-1},t^n; H^2\times H^1)}^2 + h^4 \Bigr)
\end{align*}
by setting $\beta_i = \nu/(4 \alpha_1^2)~(i=1,\cdots,4)$.
From the discrete Gronwall's inequality there exists a positive constant~$c_4$ independent of $h$ and $\Delta t$ such that
\begin{align*}
\|e_h\|_{l^\infty(L^2)} & \le c_4 ( \|e_h^0\|_0 + \Delta t + h^2 ).
\end{align*}
Using \eqref{ieq:StokesPrj_L2}, we have
\begin{align*}
\|e_h^0\|_0 &\le \|u_h^0-u^0\|_0 + \|u^0-\hat{u}_h^0\|_0 \le 2\alpha_{32} h^2 \|(u^0,p^0)\|_{H^2\times H^1},\\
\|u_h-u\|_{l^\infty(L^2)} & \le \|e_h\|_{l^\infty(L^2)} + \|\eta\|_{l^\infty(L^2)} \le \|e_h\|_{l^\infty(L^2)} + \alpha_{32}h^2 \|(u,p)\|_{C(H^2\times H^1)}. 
\end{align*}
Combining these three inequalities together, we get~\eqref{ieq:L2}.
\end{proof}
%
%
%
%
%
%
%
%
%
%
%
%
%
%
%
%
%
%
%
%%%%%%%%%%%%%%%%%%%%%%%%%%%%%%%%%%%%%%%%%%%%%
\section{Numerical results}\label{sec:numerics}
%%%%%%%%%%%%%%%%%%%%%%%%%%%%%%%%%%%%%%%%%%%%%
In this section two- and three-dimensional test problems are computed by scheme~\eqref{scheme_NS} in order to recognize the theoretical convergence orders numerically.
\par
For the computation of the integral
\begin{align*}
\int_K u^{n-1}_h\circ X_1(u_h^{n-1},\Delta t)(x) v_h(x)\ dx
\end{align*}
appearing in scheme~\eqref{scheme_NS} we employ quadrature formulae~\cite{St-1971} of degree five for $d=2$~(seven points) and $3$~(fifteen points).
The results obtained in Theorems~\ref{thm:main_results} and~\ref{thm:main_results_L2} hold  for any fixed $\delta_0$.
Here we set $\delta_0 = 1$.
The system of linear equations is solved by MINRES~\cite{Betal-1994,S-2003}.
\begin{Ex}\label{ex:test_prob}
In problem~\eqref{prob:NS} we set $\Omega=(0, 1)^d$, $T=1$ and four values of $\nu$, 
\begin{align*}
\nu=10^{-k}, \quad k=1,\cdots,4.
\end{align*}
The functions $f$ and $u^0$ are given so that the exact solution is as follows:
\medskip\\
for $d=2$
\begin{align*}
u (x,t) &= \Bigl( \prz{\psi}{x_2}, -\prz{\psi}{x_1} \Bigr)(x,t),\quad p(x,t) = \sin \{\pi (x_1 + 2 x_2 + t)\},\\
\psi(x,t) & \equiv \fz{\sqrt{3}}{2\pi} \sin^2 (\pi x_1) \sin^2 (\pi x_2) \sin \{ \pi (x_1+x_2 + t) \},
\intertext{for $d=3$}
u (x,t)  & = {\rm rot} \, \Psi (x, t),\quad
p (x,t)  = \sin \{\pi (x_1 + 2 x_2 + x_3 + t)\},\\
\Psi_1(x,t) & \equiv \fz{8\sqrt{3}}{27\pi}\sin (\pi x_1) \sin^2 (\pi x_2) \sin^2 (\pi x_3) \sin\{\pi (x_2+x_3+t)\},\\
\Psi_2(x,t) & \equiv \fz{8\sqrt{3}}{27\pi}\sin^2 (\pi x_1) \sin (\pi x_2) \sin^2 (\pi x_3) \sin\{\pi (x_3+x_1+t)\},\\
\Psi_3(x,t) & \equiv \fz{8\sqrt{3}}{27\pi}\sin^2 (\pi x_1) \sin^2 (\pi x_2) \sin (\pi x_3) \sin\{\pi (x_1+x_2+t)\}.
\end{align*}
These solutions are normalized so that $\|u\|_{C(L^\infty)}=\|p\|_{C(L^\infty)}=1$.
\end{Ex}
\par
Let $N$ be the division number of each side of the domain.
We set $N =64, 128, 256$ and $512$ for $d=2$ and $N=64$ and $128$ for $d=3$, and (re)define $h\equiv 1/N$.
Local meshes are shown in Figure~\ref{fig:Mesh} for $d=2$ (left, $N=64$, in $[0.9,1]^2$) and $3$ (right, $N=64$, in $[0.9,1]^3$).
Setting $\Delta t=\gamma_1h$ and $\gamma_2 h^2$ ($\gamma_1=4,~\gamma_2=256$), we solve Example~\ref{ex:test_prob} by scheme~\eqref{scheme_NS} with $u_h^0$, the first component of the Stokes projection of $(u^0, 0)$ by~\eqref{eq:StokesPrj}.
The two relations between $\Delta t$ and $h$, i.e., $\Delta t=\gamma_1h$ and $\gamma_2 h^2$, are employed in order to recognize the convergence orders of~\eqref{ieq:main_results} and~\eqref{ieq:L2}, respectively and we have $(\Delta t =) \gamma_1h = \gamma_2 h^2$ for $h=1/64$.
For the solution $(u_h,p_h)$ of scheme~\eqref{scheme_NS} we define the relative errors $Er1$ and $Er2$ by
\begin{align*}
Er1 \equiv \fz{\|u_h-\Pi_hu\|_{l^2(H^1)} + \|p_h-\Pi_hp\|_{l^2(L^2)}}{\|\Pi_hu\|_{l^2(H^1)}+ \|\Pi_hp\|_{l^2(L^2)}},\qquad
Er2 \equiv \fz{\|u_h-\Pi_hu\|_{l^\infty(L^2)}}{\|\Pi_hu\|_{l^\infty(L^2)}},
\end{align*}
where for the pressure we have used the same symbol $\Pi_h$ as its scalar version, i.e., $\Pi_h: C(\bar{\Omega}) \to M_h$.
Figure~\ref{fig:Error} shows the graphs of $Er1$ versus $h$ for $d=2$ and $3$ (left, $\Delta t=\gamma_1h$) and $Er2$ versus $h$ for $d=2$ (right, $\Delta t=\gamma_2h^2$) in logarithmic scale, where the symbols are summarized in Table~\ref{table:symbol}.
The values of $Er1$, $Er2$ and the slopes are presented in Table~\ref{table:Error}.
We can see that $Er1$ is almost of first order in $h$ for both $d = 2$ and $3$ and that $Er2$ is almost of second order in $h$.
These results are consistent with Theorems~\ref{thm:main_results} and~\ref{thm:main_results_L2}.
\begin{figure}[!htbp]
\centering
\includegraphics[height=4cm,clip]{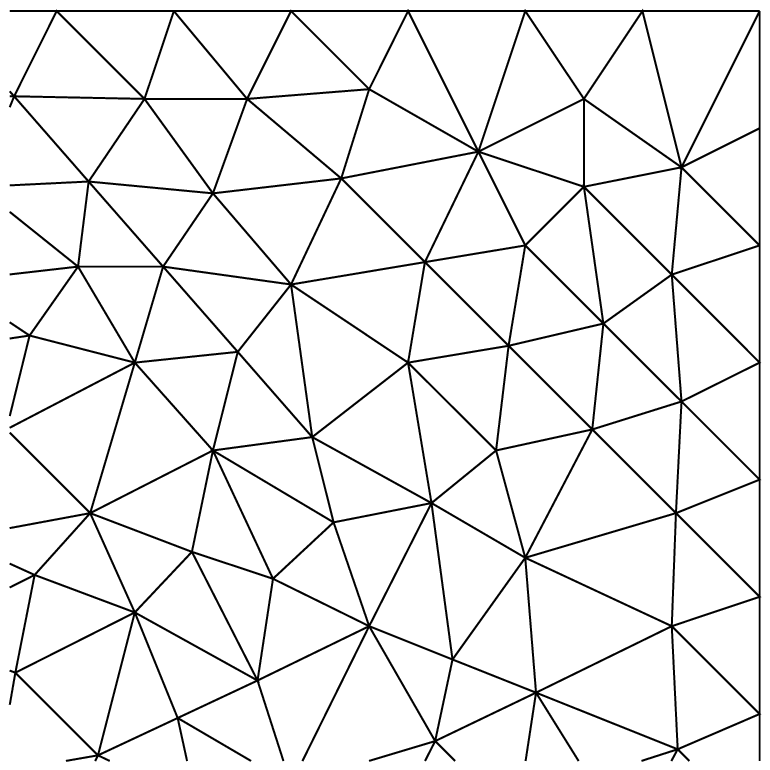}
\quad
\includegraphics[height=4cm,clip]{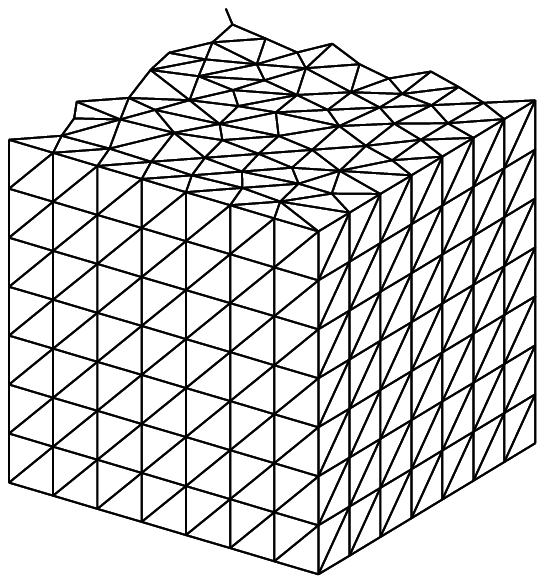}
\caption{Local meshes  for $d=2$~(left, $N=64$, in $[0.9,1]^2$) and for $d=3$~(right, $N=64$,  in  $[0.9,1]^3$ ).}
\label{fig:Mesh}
\end{figure}
%
%
%
%
%
%
%
%
%%%%%%%%%%%%%%%%%%%%%%%%%%%%%%%%%%%%%%%%%%%%%
\section{Conclusions}\label{sec:conclusions}
%%%%%%%%%%%%%%%%%%%%%%%%%%%%%%%%%%%%%%%%%%%%%
A combined finite element scheme with a Lagrange-Galerkin method and Brezzi-Pitk\"aranta's stabilization method for the Navier-Stokes equations proposed in~\cite{NT-2008-JSIAM,N-2008-JSCES} has been theoretically analyzed.
Convergence with the optimal error estimates of $O(\Delta t+h)$ for the velocity in $H^1$-norm and the pressure in $L^2$-norm  (Theorem~\ref{thm:main_results}) and of $O(\Delta t+h^2)$ for the velocity in $L^2$-norm  (Theorem~\ref{thm:main_results_L2}) have been proved.
The scheme has the advantages of both method, robustness for convection-dominated problems, symmetry of the resulting matrix and the small number of DOF.  
We note that it is a fully discrete stabilized LG scheme in the sense that the exact solvability of ordinary differential equations describing the particle path is not required.
In order to provide the initial approximate velocity we have introduced a stabilized Stokes projection, which works well in the analysis without any loss of convergence order.
The theoretical convergence orders have been recognized numerically by two- and three-dimensional computations in Example~\ref{ex:test_prob}.
It is not difficult to consider a fully discrete stabilized LG scheme of second order in time due to the idea of~\cite{BMMR-1997,ER-1981}, and its convergence with the optimal error estimates will be proved by extending the argument of this paper.
%
%
%
%%%%%%%%%%%%%%%%%%%%%%%%%%%%%%%%%%%%%%%%%%%%%
\section*{Acknowledgements}
%%%%%%%%%%%%%%%%%%%%%%%%%%%%%%%%%%%%%%%%%%%%%
This work was supported by JSPS (the Japan Society for the Promotion of Science) under the Japanese-German Graduate Externship (Mathematical Fluid Dynamics) and by Waseda University under Project research, Spectral analysis and its application to the stability theory of the Navier-Stokes equations of Research Institute for Science and Engineering.
The authors are indebted to JSPS also for Grant-in-Aid for Young Scientists~(B), No.~26800091 to the first author and for Grants-in-Aid for Scientific Research~(C), No.~25400212 and (S), No.~24224004 to the second author.
\begin{figure}[!htbp]
\centering
\includegraphics[height=7.8cm,clip]{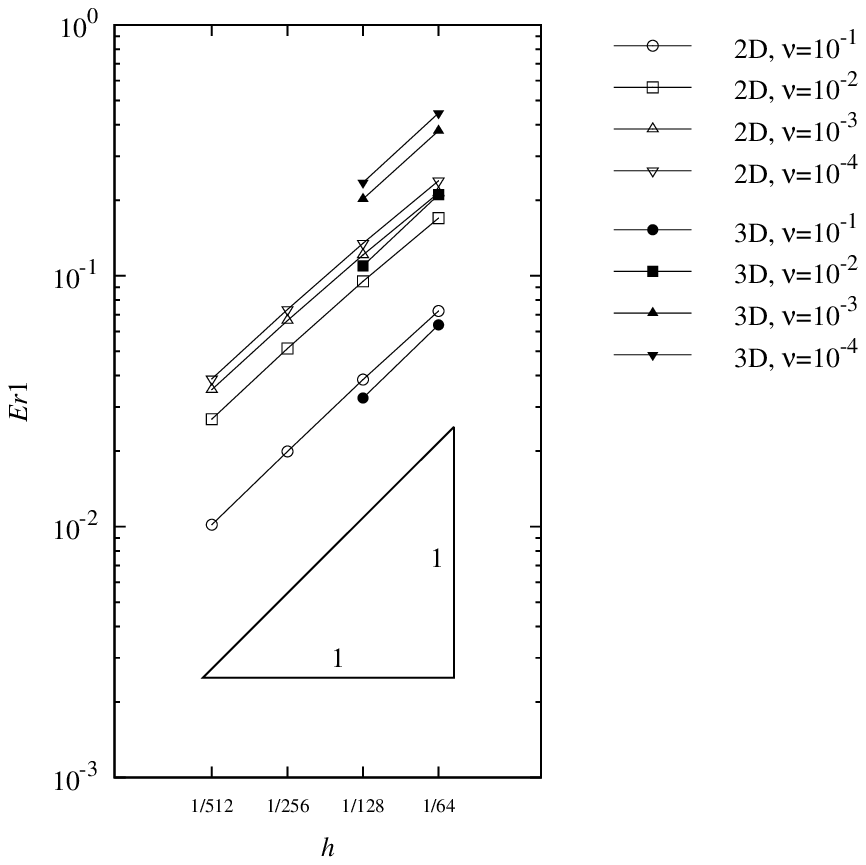}
\includegraphics[height=7.8cm,clip]{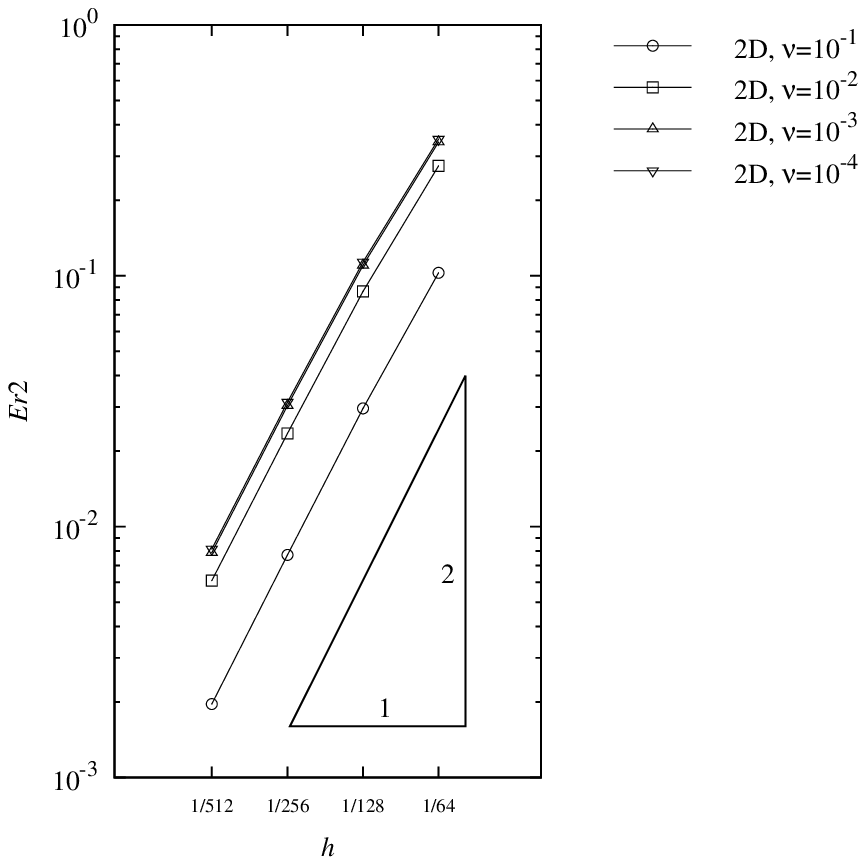}
\caption{$Er1$ vs. $h$ for $d=2$ and $3$~(left, $\Delta t= \gamma_1 h, \, \gamma_1=4$) and $Er2$ vs. $h$ for $d=2$~(right, $\Delta t=\gamma_2 h^2, \, \gamma_2=256$).}\label{fig:Error}
\end{figure}
\begin{table}[!htbp]
\caption{Symbols used in Figure~\ref{fig:Error}.}
\label{table:symbol}
\begin{center}
\begin{tabular}{ccccc}
\hline
 & \multicolumn{4}{c}{$\nu$} \\ \cline{2-5}
$d$ & $10^{-1}$ & $10^{-2}$ & $10^{-3}$ & $10^{-4}$ \\ \hline
$2$ & {\LARGE $ \circ$} & $\Box$ & $\triangle$ & {\Large $\triangledown$} \\
$3$ & {\LARGE $\bullet$} & $\blacksquare$ & {\Large $\blacktriangle$} & $\blacktriangledown$ \\
\hline
\end{tabular}
\end{center}
\bigskip
\end{table}
\begin{table}[!htbp]
\centering
\caption{Values of $Er1$, $Er2$ and slopes of the graphs in Figure~\ref{fig:Error}.}
\label{table:Error}
\begin{tabular}{lrrrcrrcrr}
\hline
&&
\multicolumn{6}{c}{$Er1$} & \multicolumn{2}{c}{$Er2$} \\ \cline{3-7} \cline{9-10}
& \multicolumn{1}{c}{$N$} & \multicolumn{1}{c}{$d=2$} & \multicolumn{1}{c}{slope} && \multicolumn{1}{c}{$d=3$} & \multicolumn{1}{c}{slope} && \multicolumn{1}{c}{$d=2$} & \multicolumn{1}{c}{slope} \\ 
\hline\hline
$\nu=10^{-1}:$
&   $64$ & $7.24\times 10^{-2}$ &       --- && $6.37\times 10^{-2}$ &      ---   && $1.03 \times 10^{-1}$ &       --- \\
& $128$ & $3.85\times 10^{-2}$ & $0.91$ && $3.25\times 10^{-2}$ & $0.97$ && $2.96 \times 10^{-2}$ & $1.80$ \\
& $256$ & $1.99\times 10^{-2}$ & $0.95$ && --- & --- && $7.71 \times 10^{-3}$ & $1.94$ \\
& $512$ & $1.01\times 10^{-2}$ & $0.97$ && --- & --- && $1.96 \times 10^{-3}$ & $1.97$ \\ \hline
$\nu=10^{-2}:$
&   $64$ & $1.70\times 10^{-1}$ &    ---    && $2.10\times 10^{-1}$ &     ---    && $2.74\times 10^{-1}$ & ---    \\
& $128$ & $9.51\times 10^{-2}$ & $0.84$ && $1.10\times 10^{-1}$ & $0.94$ && $8.66\times 10^{-2}$ & $1.66$ \\
& $256$ & $5.13\times 10^{-2}$ & $0.89$ && --- & --- && $2.35\times 10^{-2}$ & $1.88$ \\
& $512$ & $2.68\times 10^{-2}$ & $0.93$ && --- & --- && $6.09\times 10^{-3}$ & $1.95$ \\ \hline
$\nu=10^{-3}:$
&   $64$ & $2.14\times 10^{-1}$ &    ---    && $3.78\times 10^{-1}$ &    ---    && $3.41\times 10^{-1}$ & ---    \\
& $128$ & $1.21\times 10^{-1}$ & $0.82$ && $2.02\times 10^{-1}$ & $0.90$ && $1.10\times 10^{-1}$ & $1.63$ \\
& $256$ & $6.63\times 10^{-2}$ & $0.87$ && --- & --- && $3.03\times 10^{-2}$ & $1.86$ \\
& $512$ & $3.51\times 10^{-2}$ & $0.92$ && --- & --- && $7.88\times 10^{-3}$ & $1.95$ \\ \hline
$\nu=10^{-4}:$
&   $64$ & $2.39\times 10^{-1}$ &    ---    && $4.45\times 10^{-1}$ &     ---    && $3.50\times 10^{-1}$ & ---    \\
& $128$ & $1.35\times 10^{-1}$ & $0.83$ && $2.35\times 10^{-1}$ & $0.92$ && $1.13\times 10^{-1}$ & $1.63$ \\
& $256$ & $7.34\times 10^{-2}$ & $0.88$ && --- & --- && $3.13\times 10^{-2}$ & $1.85$ \\
& $512$ & $3.88\times 10^{-2}$ & $0.92$ && --- & --- && $8.14\times 10^{-3}$ & $1.94$ \\ \hline
\end{tabular}
\end{table}
%
%
%
%
%%%%%%%%%%%%%%%%%%%%%%%%%%%%%%%%%%%%%%%%%%%%%
\appendix
\renewcommand{\thesection}{A}
\setcounter{Lem}{0}
\renewcommand{\theLem}{\thesection.\arabic{Lem}}
\setcounter{Rmk}{0}
\renewcommand{\theRmk}{\thesection.\arabic{Rmk}}
\setcounter{Cor}{0}
\renewcommand{\theCor}{\thesection.\arabic{Cor}}
\setcounter{figure}{0}
\renewcommand{\thefigure}{\thesection.\arabic{figure}}
\setcounter{equation}{0}
\makeatletter
    \renewcommand{\theequation}{%
    \thesection.\arabic{equation}}
    \@addtoreset{equation}{section}
  \makeatother
%
%
%%%%%%%%%%%%%%%%%%%%%%%%%%%%%%%%%%%%%%%%%%%%%
\section*{Appendix}
%%%%%%%%%%%%%%%%%%%%%%%%%%%%%%%%%%%%%%%%%%%%%
\subsection{Proof of Lemma~\ref{lem:estimates_R}}\label{proof:lem_R}
Let $t(s)\equiv t^{n-1} +s\Delta t~(s\in [0,1])$.
We prove~\eqref{ieq:R1}.
Let $y(x, s)\equiv x-(1-s)u^{n-1}(x)\Delta t$.
It holds that 
\begin{align}
R_{h1}^n(x) 
& = \Bigl\{ \Bigl( \prz{}{t} + u^n(x)\cdot\nabla \Bigr) u \Bigr\} (x, t^n) - \fz{1}{\Delta t} \Bigl[ u\bigl( y(x,s), t(s) \bigr) \Bigr]_{s=0}^1 \notag\\
& = \Bigl\{ \Bigl( \prz{}{t} + u^{n-1}(x)\cdot\nabla \Bigr) u \Bigr\} (x, t^n) + \bigl\{ \bigl( (u^n - u^{n-1})(x)\cdot\nabla \bigr) u^n \bigr\} (x) 
- \int_0^1 \Bigl\{ \Bigl( \prz{}{t} + u^{n-1}(x)\cdot\nabla \Bigr) u \Bigr\} \bigl( y(x,s), t(s) \bigr) ds \notag\\
& = \Delta t \int_0^1ds \int_{s}^1 \Bigl\{ \Bigl( \prz{}{t} + u^{n-1}(x)\cdot\nabla \Bigr)^2 u \Bigr\} \bigl( y(x,s_1), t(s_1) \bigr) ds_1 + \Delta t \int_0^1 \Bigl\{ \Bigl( \prz{u}{t} \bigl( x, t(s) \bigr)\cdot\nabla \Bigr) u^n \Bigr\} (x) ds \notag\\
& = \Delta t \int_0^1 s_1 \Bigl\{ \Bigl( \prz{}{t} + u^{n-1}(x)\cdot\nabla \Bigr)^2 u \Bigr\} \bigl( y(x,s_1), t(s_1) \bigr) ds_1 + \Delta t \int_0^1 \Bigl\{ \Bigl( \prz{u}{t} \bigl( x, t(s) \bigr)\cdot\nabla \Bigr) u^n \Bigr\} (x) ds \notag\\
&\equiv R^n_{h11}(x) + R^n_{h12}(x). \notag
\end{align}
Each term $R_{h1i}^n $ is estimated as
\begin{subequations}\label{ieqs:Rh_11_12}
\begin{align}
\| R_{h11}^n \|_0 & \le \Delta t \int_0^1 s_1 \Bigl\| \Bigl\{ \Bigl( \prz{}{t} + u^{n-1}(\cdot)\cdot\nabla \Bigr)^2 u \Bigr\} \bigl( y(\cdot,s_1), t(s_1) \bigr) \Bigr\|_0 ds_1 
\le c_u \sqrt{\Delta t} \| u \|_{Z^2(t^{n-1},t^n)}, 
\label{ieq:R11}
\\
\| R_{h12}^n \|_0 
& \le c_u \Delta t \int_0^1 \Bigl\| \prz{u}{t} \bigl( \cdot , t(s) \bigr) \Bigr\|_0 ds 
\le c_u \sqrt{\Delta t} \Bigl\| \prz{u}{t} \Bigr\|_{L^2(t^{n-1},t^n; L^2)},
\end{align}
\end{subequations}
where for the last inequality of~\eqref{ieq:R11} we have changed the variable from $x$ to $y$ and used the evaluation $\det (\pz y(x, s_1)/\pz x) \ge 1/2~(\forall s_1\in [0,1])$ from Lemma~\ref{lem:v-vX}-(i).
From~\eqref{ieqs:Rh_11_12} we get~\eqref{ieq:R1}.
\par
\eqref{ieq:R2} is obtained as
\begin{align}
\| R_{h2}^n \|_0 
& \le \alpha_{42} \|u_h^{n-1}-u^{n-1}\|_0 \|u^{n-1}\|_{1,\infty} 
\le \alpha_{42} \|u^{n-1}\|_{1,\infty} (\|\eta^{n-1}\|_0+\|e_h^{n-1}\|_0) \label{ieq:R2_proof}\\
& \le \alpha_{42} \|u^{n-1}\|_{1,\infty} ( \alpha_{31} h\|(u,p)^{n-1}\|_{H^2\times H^1} + \|e_h^{n-1}\|_0 ).\notag
\end{align}
\par
We prove~\eqref{ieq:R3}.
Let $y(x,s) \equiv x - (1-s) u_h^{n-1}(x) \Delta t$.
Since it holds that
\begin{align*}
R_{h3}^n 
& = \fz{1}{\Delta t} \Bigl[ \eta \bigl( y(\cdot,s), t(s) \bigr) \Bigr]_{s=0}^1 = \int_0^1 \Bigl\{ \Bigl( \prz{}{t} + u_h^{n-1}(\cdot)\cdot\nabla \Bigr) \eta \Bigr\} \bigl( y(\cdot,s), t(s) \bigr) ds,
\end{align*}
we have
\begin{align*}
\| R_{h3}^n \|_0 
& \le \int_0^1 \Bigl\| \Bigl\{ \Bigl( \prz{}{t} + u_h^{n-1}(\cdot)\cdot\nabla \Bigr) \eta \Bigr\} \bigl( y(\cdot, s), t(s) \bigr) \Bigr\|_0 ds \\
& \le \int_0^1 \Bigl( \Bigl\|\prz{\eta}{t} \bigl( y(\cdot, s), t(s) \bigr) \Bigr\|_0 + \|u_h^{n-1}\|_{0,\infty} \bigl\| \nabla \eta \bigl( y(\cdot, s), t(s) \bigr) \bigr\|_0 \Bigr) ds \\
& \le \sqrt{2} \int_0^1 \Bigl\{ \Bigl\|\prz{\eta}{t} \bigl( \cdot, t(s) \bigr) \Bigr\|_0 + \|u_h^{n-1}\|_{0,\infty} \bigl\| \nabla \eta \bigl( \cdot, t(s) \bigr) \bigr\|_0 \Bigr\} ds \quad \mbox{(by Lemma~\ref{lem:v-vX}-(i))}\\
& \le \sqrt{ \fz{2}{\Delta t} } \Bigl( \Bigl\| \prz{\eta}{t} \Bigr\|_{L^2(t^{n-1},t^n; L^2)} + \|u_h^{n-1}\|_{0,\infty} \bigl\| \nabla \eta \bigr\|_{L^2(t^{n-1},t^n; L^2)} \Bigr) \\
& \le \sqrt{ \fz{2}{\Delta t} } \alpha_{31}h (\|u_h^{n-1}\|_{0,\infty} + 1) \| (u,p) \|_{H^1(t^{n-1},t^n; H^2\times H^1)},
\end{align*}
which implies~\eqref{ieq:R3}.
\par
\eqref{ieq:R4} is obtained as
\begin{align*}
\phantom{MMMMMMMMMMMa}
\| R_{h4}^n \|_0 = \fz{1}{\Delta t} \bigl\| e_h^{n-1} - e_h^{n-1} \circ X_1(u_h^{n-1},\Delta t) \bigr\|_0 \le \alpha_{40} \|u_h^{n-1}\|_{0,\infty} \|e_h^{n-1}\|_1.
\phantom{MMMMMMMMMMa}
%\qed
\end{align*}
\subsection{Proof of Lemma~\ref{lem:estimates_R_L2}}
\label{proof:lem_R_L2}
\eqref{ieq:R2_L2} is obtained by combining~\eqref{ieq:StokesPrj_L2} with~\eqref{ieq:R2_proof}.
For~\eqref{ieq:R3_L2} we divide $R_{h3}^n$ into three terms,
\begin{align*}
R_{h3}^n 
& = \ol{D}_{\Delta t} \eta^n + \fz{1}{\Delta t}\Bigl\{ \eta^{n-1} - \eta^{n-1} \circ X_1(u^{n-1},\Delta t) \Bigr\} + \fz{1}{\Delta t} \Bigl\{ \eta^{n-1}\circ X_1(u^{n-1},\Delta t) - \eta^{n-1} \circ X_1(u_h^{n-1},\Delta t) \Bigr\} \\
& \equiv R_{h31}^n + R_{h32}^n + R_{h33}^n.
\end{align*}
It holds that, in virtue of \eqref{ieq:StokesPrj_L2},
\begin{subequations}\label{ieq:R3i}
\begin{align}
\| R_{h31}^n \|_{V_h^\prime} & \le \| \ol{D}_{\Delta t}\eta^n\|_0
\le \fz{1}{\sqrt{\Delta t}} \Bigl\|\prz{\eta^n}{t}\Bigr\|_{L^2(t^{n-1},t^n; L^2)} 
\le \fz{\alpha_{32} h^2}{\sqrt{\Delta t}} \| (u,p) \|_{H^1(t^{n-1},t^n; H^2\times H^1)}, \label{ieq:R31}\\
\| R_{h32}^n \|_{V_h^\prime} & \le \alpha_{41} \|u^{n-1}\|_{1,\infty} \| \eta^{n-1} \|_0 \le \alpha_{41} \|u^{n-1}\|_{1,\infty}~\alpha_{32} h^2 \| (u,p)^{n-1} \|_{H^2\times H^1}, \label{ieq:R32}\\
\| R_{h33}^n \|_{V_h^\prime}
& = \sup_{v_h \in V_h} \fz{1}{\|v_h\|_1} \fz{1}{\Delta t} \Bigl( \eta^{n-1}\circ X_1(u^{n-1}_h,\Delta t) - \eta^{n-1} \circ X_1(u^{n-1},\Delta t), v_h \Bigr) \notag\\
& \le \sup_{v_h \in V_h} \fz{1}{\|v_h\|_1} \fz{1}{\Delta t} \bigl\| \eta^{n-1}\circ X_1(u^{n-1}_h,\Delta t) - \eta^{n-1} \circ X_1(u^{n-1},\Delta t) \bigr\|_{0,1} \|v_h\|_{0,\infty} \notag\\
& \le \alpha_{43} \bigl\| u^{n-1}_h - u^{n-1} \bigr\|_0 \| \eta^{n-1} \|_1 \alpha_{21} h^{-d/6} \label{ieq:R33_0}\\
& \le \alpha_{21} \alpha_{43} h^{-d/6} \| \eta^{n-1} \|_1 ( \|e_h^{n-1}\|_0 + \|\eta^{n-1}\|_0 ) \notag\\
& \le \alpha_{21} \alpha_{43} \alpha_{32} h^{1-d/6} \| (u,p)^{n-1} \|_{H^2\times H^1} \bigl( \| e_h^{n-1}\|_0 + \alpha_{32} h^2 \|(u,p)^{n-1}\|_{H^2\times H^1} \bigr) \notag\\
& \le c \| (u,p)^{n-1} \|_{H^2\times H^1} \bigl( \| e_h^{n-1}\|_0 + h^2 \|(u,p)^{n-1}\|_{H^2\times H^1} \bigr).\label{ieq:R33}
\end{align}
\end{subequations}
From~\eqref{ieq:R31}, \eqref{ieq:R32} and~\eqref{ieq:R33} we obtain~\eqref{ieq:R3_L2}.
\par
For~\eqref{ieq:R4_L2} we use the estimate of $R_{h3}^n$.
$R_{h4}^n$ is obtained by replacing $\eta^{n-1}$ with $-e_h^{n-1}$ in $R_{h32}^n + R_{h33}^n$.
Hence, from~\eqref{ieq:R32} and~\eqref{ieq:R33_0} we have
\begin{align*}
\| R_{h4}^n \|_{V_h^\prime}
& \le \alpha_{41} \|u^{n-1}\|_{1,\infty} \|e_h^{n-1}\|_0 + \alpha_{21} \alpha_{43} h^{-d/6} \| e_h^{n-1} \|_1 \bigl\| u^{n-1}_h - u^{n-1} \bigr\|_0 \\
& \le \alpha_{41} \|u^{n-1}\|_{1,\infty} \|e_h^{n-1}\|_0 + \alpha_{21} \alpha_{43} h^{-d/6} \| e_h^{n-1} \|_1 \bigl( \| e_h^{n-1}\|_0 + \alpha_{32}h^2 \|(u,p)^{n-1} \|_{H^2\times H^1} \bigr) \\
& \le c_u (1+h^{-d/6} \| e_h^{n-1} \|_1 ) \bigl( \| e_h^{n-1}\|_0 + h^2 \|(u,p)^{n-1} \|_{H^2\times H^1} \bigr),
\end{align*}
which implies~\eqref{ieq:R4_L2}.
%
%\qed
%
%
%
%
%
%
%
%
%%%%%%%%%%%%%%%%%%%%%%%%%%%%%%%%%%%%%%%%%%%%%
%%-----------------------------
%%      your bibliography
%%-----------------------------
%

%%%%%%%%%%%%%%%%%%%%%%%%%%%%%%%%%%%%%%%%%%%%%
%
\end{document}